\documentclass[11pt]{amsart}

\usepackage{amsmath,amsfonts,amsthm,bbm, amssymb}
\usepackage[cmtip, all]{xy}
\usepackage[letterpaper, left=3cm, right=3cm, top = 2.5cm, bottom = 2.5cm ]{geometry}
\usepackage{comment}

\newtheorem{thm}[equation]{Theorem}
\makeatletter
\let\c@subsection\c@equation
\makeatother
\newtheorem{prop}[equation]{Proposition}
\newtheorem{lem}[equation]{Lemma} 
\newtheorem{cor}[equation]{Corollary}

\theoremstyle{definition}

\newtheorem{defn}[equation]{Definition}
\theoremstyle{remark}
\newtheorem{remk}[equation]{Remark}
\newtheorem{remks}[equation]{Remarks}

\newtheorem{exm}[equation]{Example}
\newtheorem{exms}[equation]{Examples}
\newtheorem{notat}[equation]{Notation}
\numberwithin{equation}{section}

{\hfill$\square$\end{defn}}
{\hfill$\square$\end{remk}}
{\hfill$\square$\end{remks}}
{\hfill$\square$\end{exm}}
{\hfill$\square$\end{exms}}
{\hfill$\square$\end{notat}}

\newcommand{\thmref}{Theorem~\ref}
\newcommand{\propref}{Proposition~\ref}
\newcommand{\corref}{Corollary~\ref}

\newcommand{\lemref}{Lemma~\ref}

\newcommand{\sC}{{\mathcal C}}

\newcommand{\sF}{{\mathcal F}}

\newcommand{\sH}{{\mathcal H}}

\newcommand{\sO}{{\mathcal O}}

\newcommand{\sV}{{\mathcal V}}

\newcommand{\sZ}{{\mathcal Z}}

\newcommand{\A}{{\mathbb A}}

\newcommand{\C}{{\mathbb C}}

\newcommand{\G}{{\mathbb G}}

\renewcommand{\P}{{\mathbb P}}
\newcommand{\Q}{{\mathbb Q}}

\newcommand{\Z}{{\mathbb Z}}

\newcommand{\CH}{{\rm CH}}

\newcommand{\surj}{\twoheadrightarrow}
\newcommand{\inj}{\hookrightarrow}

\newcommand{\Spec}{{\rm Spec \,}}

\newcommand{\Sch}{{\operatorname{\mathbf{Sch}}}}

\renewcommand{\>}{\rangle}

\newcommand{\Sm}{{\mathbf{Sm}}}

\newcommand{\Ab}{{\mathbf{Ab}}}

\newcommand{\ds}{{/\kern-3pt/}}

\newcommand{\ov}{\overline}

\renewcommand{\dim}{\text{\rm dim}}

\newcommand{\tuborg}{\left\{\begin{array}{ll}}
\newcommand{\sluttuborg}{\end{array}\right.}

\newcommand{\wt}{\widetilde}
\newcommand{\wh}{\widehat}

\newcounter{elno}

\newcounter{elno-abc}   

\newcounter{elno-abc-prime}

\begin{document}
\title[Equivariant $K$-theory and Higher Chow Groups]
{Equivariant $K$-theory and Higher Chow Groups of schemes}
\author{Amalendu Krishna}
\address{School of Mathematics, Tata Institute of Fundamental Research,  
1 Homi Bhabha Road, Colaba, Mumbai, India.}
\email{amal@math.tifr.res.in}



\begin{abstract}
For a smooth quasi-projective scheme $X$ over a field $k$
with an action of a reductive group, we 
establish a spectral sequence connecting the equivariant and the 
ordinary higher Chow groups of $X$. For $X$ smooth and projective, 
we show that this spectral sequence degenerates, leading to an explicit
relation between the equivariant and the ordinary higher Chow groups. 
We obtain several applications to algebraic $K$-theory. 

We show that for a reductive 
group $G$ acting on a smooth projective scheme $X$, 
the forgetful map $K^G_i(X) \to K_i(X)$ induces an isomorphism
$K^G_i(X)/{I_G K^G_i(X)} \xrightarrow{\simeq} K_i(X)$ with rational 
coefficients. This generalizes a result of Graham to higher $K$-theory of 
such schemes. We prove an equivariant Riemann-Roch theorem, leading to 
a generalization of a result of Edidin and Graham to higher $K$-theory. 
Similar techniques are used to prove the equivariant Quillen-Lichtenbaum 
conjecture.
\end{abstract} 
\setcounter{tocdepth}{1}
\maketitle
  
\tableofcontents

\section{Introduction}\label{sec:Intro}
Based on the ideas of Morel-Voevodsky \cite{MV} and
Totaro \cite{Totaro} for the Borel construction of the classifying spaces
of linear algebraic groups, the theory of equivariant higher Chow groups 
was formalized in the present form in \cite{ED2}.
Totaro studied the Chow ring of the algebraic classifying spaces of linear 
algebraic groups in \cite{Totaro}.
His model for the classifying space was also defined independently by 
Morel-Voevodsky \cite{MV} to study the motivic cohomology of 
infinite dimensional
schemes (represented in the $\A^1$-homotopy category as motivic sheaves). 
The Borel style equivariant motivic cohomology groups were 
used by Voevodsky \cite{Voev-1} (without highlighting their
formal definition) in his work on the reduced power operations in motivic
cohomology.  

There has been a considerable amount of work in recent past to understand the 
equivariant Chow groups and to relate them with the equivariant $K$-theory. 
One of the motivations in the study of equivariant
higher Chow groups is the study of their relation with the ordinary 
(non-equivariant) higher Chow groups of schemes. 
In particular, one would like to know if
there are tools which can be used to study one in terms of the other.

It turns out that compared to their non-equivariant companions,
the equivariant $K$-theory or the Chow groups
are often easier to compute 
(see, for instance, \cite{Brion} and \cite{Krishna1}).
Hence, any formula relating these two would be helpful in the study of the 
ordinary higher $K$-theory and Chow groups of schemes with group action. 
The goal of this paper is to prove some results in this direction
and provide several important applications to higher 
$K$-theory.

\subsection{The main results}\label{sect:MR*}
Let $k$ be a field.
Let $G$ be a smooth linear algebraic group over $k$ which acts linearly 
(see \S~\ref{sec:Note}) on
an equi-dimensional quasi-projective $k$-scheme $X$. For $i \ge 0$,
let $\CH^*_G(X,i) = {\underset{j \ge 0}\oplus} \ \CH^j_G(X,i)$ denote  
the equivariant higher Chow groups of $X$.
These groups and their basic functorial properties
are recalled in \S~\ref{sect:Recall} of this text.
The ring $S(G) := \CH^*_G(\Spec(k),0)$ is called the Chow ring of the
classifying space of $G$. This is a $\Z$-algebra with an augmentation
$S(G) \surj \Z$. One of the properties of equivariant Chow groups
is the forgetful map $\CH^*_G(X,i) \to \CH^*(X,i)$ from the equivariant to the 
ordinary higher Chow groups (see ~\eqref{eqn:forgetful}).

For a split reductive group $G$ over $k$, we write $t_G$ for its
torsion index \cite{Groth}. The torsion index is a positive integer, 
which is equal to one if $G$ is special in the sense of Grothendieck--Serre, 
i.e., all \'etale locally trivial $G$-torsors are Zariski locally trivial.
The special reductive groups over algebraically closed fields were
classified by Grothendieck and Serre. Over arbitrary fields,
these groups were recently classified by Huruguen \cite{Huru}.

The main results of this text can be stated as follows.
The underlying notations and definitions can be found in the body of the
text.

\begin{thm}\label{thm:SS-Red-Main-I}
Let $G$ be a split reductive group acting on a smooth
quasi-projective scheme $X$ over $k$. 
Then there is a convergent homological spectral sequence 
\begin{equation}\label{eqn:SS-M-01} 
E^2_{p,q} = {\rm Tor}^{S(G)}_p({\Z}, \CH^*_G(X, q))[t^{-1}_G]
\Rightarrow \CH^*(X, p+q)[t^{-1}_G]
\end{equation}
such that its edge homomorphism $\CH^*_G(X,q)
{\otimes}_{S(G)} {\Z}[t^{-1}_G] \to \CH^*(X, q)[t^{-1}_G]$
is induced by the forgetful map. 
\end{thm}

The above spectral sequence yields the algebraic analogue
of the topological Eilenberg-Moore spectral sequence \cite{EM}
\begin{equation}\label{eqn:EMoore}
E^2_{p,q} = {\rm Tor}^{H^*_G(pt)}_{p,q}(\Z, H^*_G(X, \Z)) \Rightarrow
H^{p+q}(X, \Z)
\end{equation}
in equivariant singular cohomology. 

To see the connection between the two spectral sequences, recall that for 
a topological space $X$ with an action of a topological group $G$, there is a
topological fibration $X \to X \stackrel{G}{\times} EG \to BG$,
where $BG$ is the classifying space of $G$ and $EG \to BG$ is a 
universal fiber bundle for the group $G$. An application of the
Eilenberg-Moore machinery \cite{EM} to this fibration yields 
~\eqref{eqn:EMoore}.

On the algebraic side, we have the $\A^1$-homotopy theory of 
Morel-Voevodsky \cite{MV}, where the notion of fibrations of motivic sheaves
makes sense. Given a reductive group $G$ over $k$, there is a motivic 
sheaf $BG$ as an object of the $\A^1$-homotopy category $\sH(k)$.
Similarly, given an action of $G$ on a smooth scheme $X$, there is a
motivic sheaf $X \stackrel{G}{\times} EG$ in $\sH(k)$ and
there is morphism $X \stackrel{G}{\times} EG \to BG$ with fiber $X$.
Theorem~\ref{thm:SS-Red-Main-I} then says that in this algebraic set up too,
the Eilenberg-Moore type spectral sequence exists.


Let $G$ be a reductive group acting on a quasi-projective
$k$-scheme $X$. If $X^G$ denotes the fixed point locus for this action, 
we have $X^G \stackrel{G}{\times} EG \simeq X^G \times BG$. 
This gives rise to a diagram
\[
\xymatrix@C1pc{
\CH^*(X^G, q) \otimes_{\Z} S(G) \ar[r] \ar[d] & \CH^*_G(X^G, q) \ar[d] \\
\CH^*(X, q) \otimes_{\Z} S(G) \ar@{-->}[r] & \CH^*_G(X, q),}
\]
where the horizontal arrow on the top is the external product map and the
vertical arrows are the push-forward maps.
Our next result shows that the dotted arrow can be completed to
an isomorphism in certain cases.

\begin{thm}\label{thm:Degn-*}
Let $G$ be a split reductive group acting on a smooth projective
scheme $X$ over $k$ and let $q \ge 0$ be any integer.
\begin{enumerate}
\item
If $G$ is a split torus, there is an isomorphism of $S(G)$-modules
\[
\CH^*(X, q) {\otimes}_{\Z} S(G) 
\xrightarrow{\simeq} \CH^*_G(X, q).
\]
\item
The natural map
\[
r^G_X: \CH^*_G(X,q) \otimes_{S(G)} \Z[t^{-1}_G] \to \CH^*(X,q)[t^{-1}_G]
\]
is an isomorphism.
\end{enumerate}
\end{thm}

\subsection{Applications}\label{sect:Appl-*}
We now give some important applications of the above results.
As the first consequence of Theorem~\ref{thm:Degn-*}, we prove the following
generalization of the results of Merkurjev
\cite[Corollary~10]{Merkurjev} and Graham \cite[Theorem~1.1]{Graham} to 
higher $K$-theory of smooth projective schemes.

For a smooth linear algebraic group $G$, let $R(G)$ denote the representation
ring of $G$. This is another name of the Grothendieck group $K^G_0(k)$
of $G$-equivariant vector bundles on $\Spec(k)$. Let
$I_G$ denote the ideal of $R(G)$ which consists of the virtual representations
of $G$ of rank zero. That is, $I_G$ is the kernel of the rank map
$R(G) \to \Z$. For a $G$-scheme $X$, let $K^G_*(X)$ 
(resp. $G^G_*(X)$) denote the homotopy groups of the $K$-theory 
spectrum of $G$-equivariant vector bundles (resp. coherent sheaves) on $X$.
Let ${\widehat{K^G_i(X)}}$ denote the 
completion of the $R(G)$-module $K^G_i(X)$ with respect to the augmentation 
ideal $I_G$.

\begin{thm}\label{thm:forgetful}
Let $G$ be a reductive group acting on a smooth 
projective scheme $X$ over $k$. Then for all $i \ge 0$, the forgetful map 
$K^G_i(X) \to K_i(X)$ induces an isomorphism
\begin{equation}\label{eqn:char}
{K^G_i(X)}/{I_GK^G_i(X)} \xrightarrow{\simeq} K_i(X)
\end{equation}
with rational coefficients. In particular, the natural map
${K^G_i(X)}_{\Q} \to {K_i(X)}_{\Q}$ is surjective.
\end{thm}

Graham had proven this result for the $K_0$ of any scheme with an action of 
a reductive group. Merkurjev proved this for $K_0$ 
(with integral coefficients) for those groups whose commutator subgroups
are simply connected. In general however, it is necessary to tensor the 
$K$-groups with $\Q$ to prove the surjectivity of the forgetful map 
(see \cite[Example~4.1]{Graham}). 

We shall show that (see Example~\ref{exm:Forget-2})
\thmref{thm:forgetful} no longer holds if $X$ is a smooth quasi-projective
$k$-scheme which is not necessarily projective, and if $i \neq 0$.

Our second application of Theorems~\ref{thm:SS-Red-Main-I} and 
~\ref{thm:Degn-*} is
the following generalization of the equivariant Riemann-Roch theorem 
of Edidin and Graham \cite[Theorem~4.1]{ED1} to higher 
$K$-theory. This also gives a refinement of the more general 
equivariant Riemann-Roch theorems of \cite{KrishnaV}
for smooth projective schemes and smooth quasi-projective toric varieties. 

\begin{thm}\label{thm:RRoch-Main}
Let $G$ be a reductive group acting on a smooth quasi-projective
scheme $X$ over $k$. Assume one of the following.
\begin{enumerate}
\item
$X$ is projective.
\item
$X$ is a toric variety and $G$ is the dense torus of $X$.
\item
$G = \G_m$.
\end{enumerate}
Then for every $i \ge 0$, there is a Riemann-Roch isomorphism 
\[
{\tau}^G_X : {\widehat{K^G_i(X)}}_{\Q} \xrightarrow{\simeq} 
\stackrel{\infty}{\underset{j= 0}{\prod}} {\CH^j_G(X, i)}_{\Q}
\]
\end{thm}

As an analogue of \thmref{thm:RRoch-Main} for the equivariant $K$-theory
with finite coefficients, we prove the following
equivariant version of the Quillen-Lichtenbaum conjecture.
For a $\C$-scheme $X$, let $X^{an}$ denote the associated analytic space
$X(\C)$.

\begin{thm}\label{thm:QL-*}
Let $G$ be a reductive group acting on a smooth quasi-projective
scheme $X$ of dimension $d$ over $\C$. Let $M$ be a maximal compact 
subgroup of the Lie group $G^{an}$. Let $p$ be a prime number.
Then the natural map
\[
{\rho}_X : K^G_i(X; {\Z}/{p^{\nu}}) \to
K^M_i(X^{an}; {\Z}/{p^{\nu}})
\]
is an isomorphism for $i \ge d-1$ and a monomorphism for $i = d-2$.
\end{thm}

For other applications, we refer the reader to \cite[Chap.~16]{Totaro-2},
where the spectral sequence ~\eqref{eqn:SS-M-01} for the
action of split tori was used by Totaro to independently
construct an Eilenberg-Moore spectral sequence for the
equivariant higher Chow groups for actions of split reductive groups.
This spectral sequence was then used by him to compute the 
motivic cohomology of spaces like $GL_n/G$, given a faithful
representation $G \to GL_n$ of a reductive group.

\begin{remk}\label{remk:Reductive}
In all the results stated above, it is assumed that the underlying group
is split reductive. However, these results hold for the action of any  
smooth connected linear algebraic group $G$, 
provided the ground field $k$ is perfect
and the reductive quotient of $G$ is split.
The reason for this is that any smooth connected
linear algebraic group over a perfect 
field has a smooth unipotent radical. Since a smooth unipotent group over a 
perfect field admits a composition series with successive subquotients
isomorphic to the additive group, the homotopy invariance reduces
the general case to the reductive case.
\end{remk} 

\subsection{Brief outline of this paper}
Our construction of the spectral sequence is based on a technique 
developed by Levine \cite{Levine} to generalize  Quillen's localization 
sequence in algebraic $K$-theory.  
In \S~\ref{sect:SS-I}, we briefly recall the definition and basic
properties of the equivariant higher Chow groups. We also recall Levine's
notion of $n$-cubes in a category and explain as to 
how these $n$-cubes give rise
to a tower of simplicial abelian groups and the resulting 
Bousfield-Kan spectral sequence. 

\S~\ref{sect:SS-T*} is devoted to the construction of the claimed 
Eilenberg-Moore spectral sequence in the case
of a split torus action. To do this, we adapt Levine's technique to
produce a spectral sequence with an easily described $E_1$-page.
By combining this spectral sequence with some geometric inputs,
we are able to compute the $d_1$-differentials and identify the $E_2$-page
explicitly. The general case is proven in \S~\ref{sect:SS-R} by reducing to 
the case of maximal tori of split reductive groups. This reduction
is achieved through \thmref{thm:Sspl-M}.

In \S~\ref{sect:Proj}, we use the results of Bialynicki-Birula
to obtain an integral presentation of the equivariant higher Chow groups of 
smooth projective schemes with a split torus action. This presentation is
used to prove \thmref{thm:Degn-*}. In \S~\ref{sect:EKT}, we give applications 
of Theorems~\ref{thm:SS-Red-Main-I} and ~\ref{thm:Degn-*} to the equivariant 
$K$-theory. The final section 
consists of a proof of the equivariant Quillen-Lichtenbaum conjecture.

\section{Equivariant Chow 
Groups and Bousfield-Kan Spectral Sequence}\label{sect:SS-I}
The main tool we shall use to construct the equivariant to ordinary higher 
Chow groups spectral sequence is a technique developed by
Levine \cite{Levine} to generalize Quillen's
localization sequence in algebraic $K$-theory.
We shall construct a spectral sequence based on Levine's idea
and then prove \thmref{thm:SS-Red-Main-I} by suitably 
identifying the terms of this spectral sequence. 
\subsection{Basic Notations}\label{sec:Note}
We shall fix a base field $k$
and let $\Sch_k$ denote the category of separated schemes which are of
finite type over $k$. The full subcategory of $\Sch_k$ consisting of
smooth schemes over $k$ will be denoted by $\Sm_k$.
We shall assume all linear algebraic groups to be smooth over $k$. 
Recall that an action of a 
linear algebraic group $G$ on $X \in \Sch_k$ is said 
to be {\sl linear}, if $X$ admits a $G$-equivariant ample line bundle.
When $G$ is connected and $X$ is normal and quasi-projective over $k$, 
Sumihiro \cite[Theorem~2.5]{Sumihiro} showed that any $G$-action on $X$ 
can be linearized. This assumption of connectedness was
later removed by Thomason \cite[5.7]{Thomason1}. 
But it was essential for Thomason's argument
that $X$ be still quasi-projective. 
Apart from this, we also need
to use Bloch's localization sequence for higher Chow groups, which
requires one to assume the underlying scheme to be quasi-projective.

We shall therefore assume throughout this text that 
all schemes in $\Sch_k$ with
$G$-action are quasi-projective over $k$ with linear $G$-action.
We shall denote the category of such schemes by $\Sch^G_k$. 
The full subcategory of $\Sch^G_k$ consisting of smooth $k$-schemes
will be denoted by $\Sm^G_k$. 
Recall that a reductive group $G$ is {\sl split} (over $k$),
if it admits a maximal torus which is split over $k$ 
(i.e., is a product of some copies of $\G_m/k$).
In this text, we shall assume all representations of $G$ to be
rational and finite-dimensional.

\enlargethispage{25pt}

\subsection{Equivariant Chow groups}
\label{sect:Recall}
Let $G$ act on an equi-dimensional scheme $X$ over $k$
and let $j \ge 0$ be an integer.
A {\sl good pair} for $G$-action (corresponding to $j$)
is a rational representation $V$ of $G$ and a $G$-invariant
open $U \subseteq V$ such that the following hold.
\begin{enumerate}
\item
$G$ acts freely on $U$ such that the quotient $U/G$ is 
smooth and quasi-projective over $k$.
\item
The complement
$V \setminus U$ has codimension bigger than $j$. 
\end{enumerate}

It is known that the good pairs always exist \cite[Lemma~9]{ED1}.
Moroever, the mixed quotient $X \stackrel{G}{\times} U$ is a 
quasi-projective scheme (see \cite[Proposition~23]{ED2}, 
\cite[Lemma~B.1]{KrishnaV}) and is smooth over $k$ if $X$ is so. 
We shall often denote this mixed quotient by $X_G$, if the pair $(V,U)$
is understood.
The equivariant higher Chow group $\CH^j_G\left(X, i \right)$ is 
defined as the homotopy group $H_i({\sZ}^j(X_G, \bullet))$ of the 
simplicial abelian group ${\sZ}^j(X_G, \bullet)$ 
whose associated chain complex via the Dold-Kan correspondence is
the Bloch's cycle complex of the scheme $X_G$. 
The groups $\CH^j_G(X, i)$ are independent of the choice of a good pair 
$(V,U)$. We let 
\[
\CH^*_G(X,i) = \oplus_{j \ge 0} \CH^j_G(X, i), \
\CH^*_G(X,\bullet) = \oplus_{i \ge 0} \CH^*_G(X, i) \
\mbox{and} \ \CH^*_G(X) = \CH^*_G(X, 0).
\]

Let $J_G$ denote the kernel of the augmentation map $S(G) \surj \Z$,
where $S(G) = \CH^*(BG) := \CH^*_G(\Spec(k))$
coincides with Totaro's Chow group of the classifying space $BG$ 
\cite{Totaro}.
The following result summarizes most of the essential properties of 
the equivariant higher Chow groups that are used in this text. 

\begin{prop}$($\cite[Proposition~2.2]{Krishna1}$)$\label{prop:EHCG}
The equivariant higher Chow groups as defined above satisfy the following
properties. \\
\begin{enumerate}
\item
$Functoriality:$ Covariance for proper maps, contravariance for
flat maps and their compatibility. That is, for a Cartesian square
\[
\xymatrix@C.7pc{
X' \ar[r]^{g'} \ar[d]_{f'} & X \ar[d]^{f} \\
Y' \ar[r]_{g} & Y}
\]
in $\Sch^G_k$ with $f$ proper and $g$ flat, 
one has $g^* \circ f_* = {f'}_* \circ
{g'}^* : \CH_G^*(X, i) \to \CH_G^*(Y', i)$.
Moreover, if $f:X \to Y$ is a morphism in $\Sch^G_k$ with
$Y$ in $\Sm^G_k$, then there is a pull-back map $f^*: \CH_G^*(Y,i)
\to \CH_G^*(X,i)$. All pull-back and push-forward maps are $S(G)$-linear.
\item
$Homotopy:$ If $f:X \to Y$ is an equivariant affine bundle, then
$f^*: \CH_G^*(Y,i) \xrightarrow{\simeq} \CH_G^*(X,i)$. 
\item
$Exterior \ product:$ There is a natural product map 
\[
\CH_G^*(X,i) {\otimes}  \CH_G^*(Y,i') \to  \CH_G^*(X \times Y,i+i').
\]
Moreover, if $f: X \to Y$ is such that $Y \in \Sm^G_k$, then there is
a pull-back via the graph map ${\Gamma}_f : X \to {X \times Y}$, which makes
$\oplus_{i \ge 0} \CH^*_G(Y,i)$ a bigraded ring and 
$\oplus_{i \ge 0}  \CH_G^*(X)$ a module
over this ring. In particular, $\CH_G^*(X, i)$ an $S(G)$-module for 
$X \in \Sch^G_k$ and $i \ge 0$.
\item
$Localization:$ If $Y \subset X$ is a $G$-invariant closed
subscheme with complement $U$, then there is a long exact localization 
sequence of $S(G)$-modules
\[
\cdots \to  \CH_G^*(Y,i) \to \CH_G^*(X,i) \to  \CH_G^*(U,i) \to
\CH_G^*(Y,i-1) \to \cdots.
\]
This sequence is compatible with the proper push-forward and flat pull-back maps
of higher Chow groups.
\item
$Chern  \ classes:$ For any $G$-equivariant vector bundle of rank $r$,
there are equivariant Chern classes $c^G_l(E):  \CH_G^j(X,i) \to
 \CH_G^{j+l}(X,i)$ for $0 \le l \le r$, having the same functoriality 
properties as in the non-equivariant case and $c^G_0(E) = 1$. 
\item
$Projection \ formula:$ For a proper map $f: X \to Y$ in
$\Sch^G_k$ with $Y$ smooth, and for $x \in \CH_G^*(X),
y \in \CH^{*}_G(Y)$, one has $f_*\left(f^*(y) \cdot x \right) =
y \cdot f_*(x)$. 
\item
$Free \ action:$ If $G$ acts freely on $X$ with quotient $Y \in \Sch_k$, then 
there is  a canonical isomorphism $\CH_G^*(X,i) \xrightarrow{\simeq}
\CH^*(Y,i)$.
\end{enumerate}
\end{prop}

\begin{remk}\label{remk:shift}
The reader should be warned that various isomorphisms between the
(equivariant) higher Chow groups in the above proposition are true only up to 
some obvious shift in the dimension of cycles, which we chose not to
highlight.

Another remark is that 
part (2) of the proposition is proven in \cite[Proposition~2.2]{Krishna1} only
for equivariant vector bundles. But the same proof works more generally
for any equivariant affine bundle. Recall here that a $G$-equivariant
affine bundle over $X \in \Sch^G_k$ is a morphism $\phi: Y \to X$ in
$\Sch^G_k$, where $f: V \to X$ is a $G$-equivariant vector bundle
and $\phi$ is a torsor under $V$ so that $\phi$ and the action map 
$V \times_X Y \to Y$ are both $G$-equivariant. 
Recall also that
an $\A^n$-fibration in $\Sch_k$ is a flat morphism $\pi: Y \to X$
such that the scheme-theoretic fiber over every point 
(not necessarily closed) $x \in X$
is isomorphic to $\A^n_{k(x)}$. One says that such a fibration is
locally trivial in some topology $\tau$ on $\Sch_k$ if there is a
$\tau$-covering $X' \to X$ such that $X' \times_X Y \simeq X' \times \A^n_k$.

Now, given a $G$-equivariant
affine bundle $\phi: Y \to X$, a good pair $(V,U)$ for $G$ gives rise to a 
$G$-equivariant affine bundle
$\phi_U: Y \times U \to X \times U$. In turn, this yields an
$\A^n$-fibration 
$Y \stackrel{G}{\times} U \to X \stackrel{G}{\times} U$ 
for some integer $n \ge 0$ (which is locally trivial in the 
$fppf$-topology). 
On the other hand, the invariance of the ordinary higher Chow
groups under $\A^n$-fibration is well known and is an easy consequence
of their nil-invariance, localization  and  
homotopy invariance (for trivial $\A^n$-fibration) properties, and 
the noetherian induction.
\end{remk}

We recall recall from \cite{ED2} that if $H \subset G$ is a closed 
subgroup and if $(V,U)$ is a good pair, then for $X \in {\sV}_G$, 
the natural map of 
quotients $X \stackrel{H}{\times} U \to X \stackrel{G}{\times} U$ is a
smooth morphism (an \'etale locally trivial $G/H$-fibration) 
and hence there is a natural restriction map
\begin{equation}\label{eqn:res}
r^G_{H, X} : \CH_G^*\left(X,i\right) \to \CH_H^*\left(X,i\right).
\end{equation} 
Taking $H = \{1\}$, one obtains the {\sl forgetful} map
\begin{equation}\label{eqn:forgetful}
r^G_X : \CH_G^*\left(X,i\right) \to \CH^*\left(X, i\right).
\end{equation}
Moreover, as $r^G_{H,X}$ is the pull-back under a flat (in fact, a smooth)
map, it commutes
(see Proposition~\ref{prop:EHCG}) with the pull-back for any flat 
map, and with the push-forward for any proper map in $\Sch^G_k$. 
We remark here that although the definition of $r^G_{H,X}$ uses a good pair
$(V,U)$, it is easy to check from the homotopy
invariance that $r^G_{H,X}$ is independent of the choice of the good pair
$(V,U)$.

The following are the analogues for equivariant Chow groups 
of the Morita isomorphisms in equivariant $K$-theory.

\begin{prop}$($\cite[Corollary~3.2]{Krishna1}$)$\label{prop:Morita}
Let $H \subset G$ be a closed subgroup and let $X \in \Sch^H_k$.
Then for any $i \ge 0$, there is a natural isomorphism
\begin{equation}\label{eqn:MoritaI}
\CH^*_G(G \stackrel{H}{\times} X, i) \xrightarrow{\simeq}
\CH^*_H\left(X, i\right).
\end{equation}
\end{prop}

\begin{prop}$($\cite[Proposition~3.3]{Krishna1}$)$\label{prop:Borel}
Let $G$ be a reductive group over $k$. Let $B$ be a Borel 
subgroup of $G$ containing a maximal torus $T$ over $k$. Then for any
$i \ge 0$ and $X \in \Sch^B_k$, the restriction map
\begin{equation}\label{eqn:Borel2}
\CH^*_B\left(X, i\right) \xrightarrow{r^B_{T,X}} \CH^*_T\left(X, i\right)
\end{equation}
is an isomorphism.
\end{prop}

\begin{thm}$($\cite[Theorem~8.5]{Krishna1}$)$\label{thm:reductiveI} 
Let $G$ be a reductive group and let $T$ be a split maximal
torus of $G$. Then for any $X \in \Sch^G_k$ and $i \ge 0$, the natural map of 
$S(T)$-modules
\begin{equation}\label{eqn:redI*0}
\lambda_X: \CH^*_G\left(X,i\right) {\underset{S(G)}\otimes} S(T) \to
\CH^*_T\left(X,i\right)
\end{equation}
is an isomorphism with rational coefficients. This is a ring isomorphism if 
$X \in \Sm^G_k$ and we take sum over all $i \ge 0$.
\end{thm}

\subsection{$n$-cubes and associated spectral sequence}
\label{sect:SSCube}
Recall that the $n$-cube $\<n\>$ is the category of subsets
(including the empty set) of the set $\{1, \cdots , n\}$ with morphisms being
inclusions of subsets. An $n$-cube $X_*$ in a category $\sC$ is a covariant
functor $X : \<n\> \to \sC$. For any $I \in \<n\>$, we denote the object 
$X(I)$ by $X_I$. For $n \ge 2$, we let $X_{*/n}$ denote the restriction
of $X_*$ to the subsets of $\{1, \cdots , n-1\}$ and let $X_{* \supset n}$
denote the restriction of $X_*$ to the subsets of $\{1, \cdots , n\}$
containing $n$. 
The inclusion $I \inj I \cup \{n\}$ defines a natural transformation
of $(n-1)$-cubes $r(X_*): X_{*/n} \to X_{* \supset n}$.

Let $\Delta_{\Ab}$ denote the category of simplicial abelian groups.
Recall from \cite[Theorem~2.8]{GJ} that $\Delta_{\Ab}$ admits a
closed model structure in which weak equivalences are isomorphisms
of homotopy groups, fibrations are the Kan fibrations of underlying
simplicial sets and the cofibrations are given by the left lifting
property with respect to maps which are fibrations and weak
equivalence. 

If we let ${\rm Ch}^{-}_{\Ab}$ denote the category of
bounded above chain complexes of abelian groups and define weak equivalence
to be quasi-isomorphism, fibration to be levelwise surjection, then
we get a closed model structure on ${\rm Ch}^{-}_{\Ab}$.
The Dold-Kan correspondence gives a Quillen
equivalence of model categories ${\rm Ch}^{-}_{\Ab} \simeq \Delta_{\Ab}$,
and hence an equivalence of their homotopy categories 
(see \cite[Chap.~III, \S~2]{GJ} for details).
We shall make no further distinction between these homotopy categories
for the rest of this text, and denote this common homotopy category by
$D^{-}(\Ab)$. 

Given an $n$-cube $X_*$ in $\Delta_{\Ab}$, we define an object 
${\rm fib}X_*$ in $\Delta_{\Ab}$, which is functorial in cubes,
as follows. For $n= 1$, we take ${\rm fib}X_*$ to be the homotopy fiber
of the obvious map $X_{\emptyset} \to X_{\{1\}}$. 
For $n \ge 2$, we inductively define ${\rm fib}X_*$ to be
the homotopy fiber of the natural map 
${\rm fib}(r(X_*)):{\rm fib} X_{*/n} \to {\rm fib} X_{* \supset n}$ of the 
homotopy fibers of $(n-1)$-cubes. 
Given an $n$-cube $X_*$ in $\Delta_{\Ab}$ and $p \ge 0$, let $X_{* \ge p}$
denote the $n$-cube gotten from $X_*$ by replacing the simplicial abelian group
$X_I$ with the zero object $\{*\}$ of $\Delta_{\Ab}$ whenever $|I| < p$.
This gives us a tower of $n$-cubes
\[
\{*\} \to X_{* \ge n} \to X_{* \ge n-1} \to \cdots \to X_{* \ge 0} = X_*.
\]

Associated to this tower, one gets a Bousfield-Kan type spectral
sequence \cite[Chap.~X, \S~6]{BKan} (see also \cite[\S~~1]{Levine})  
\begin{equation}\label{eqn:LSS}
E(X_*) : \ \ E^1_{p,q} = {\underset{|I| = p}{\oplus}}
{\pi}_{-q}(X_I) \Rightarrow {\pi}_{-q-p}\left({\rm fib} X_*\right),
\end{equation}
where the $E^1$-differential is the alternating sum of the maps 
$\pi_*(X_{I \subset I \cup j})$ between the homotopy groups,
induced by 
\[
X_{I \subset I \cup j} : X_I \to X_{I \cup j}, \ j \notin I.
\]
In particular, for any $e_I \in \pi_{-q}(X_I)$, we have
\begin{equation}\label{eqn:LSS0}
d^1_{p,q}(e_I) = {\underset{i_j \notin I}
{{\underset{1 \le i_1 < \cdots < i_{n-p} \le n}{\sum}}}} {(-1)}^{j-1} {\pi}_{-q}
(X_{I \subset I \cup i_j}) (e_I).
\end{equation}

\section{Proof of Theorem~\ref{thm:SS-Red-Main-I}: the torus 
case}\label{sect:SS-T*}
Let $T$ be a split torus over $k$ of rank $r \ge 1$.
In order to prove Theorem~\ref{thm:SS-Red-Main-I} for $T$-action,
our approach is to look at the inclusion $T \inj \A^r_k$ and filter its 
complement by $T$-invariant linear subspaces. A consideration of the
cycle complexes associated to these subspaces gives rise to 
an $r$-cube in $\Delta_{\Ab}$. We complete the proof of the theorem by 
suitably identifying the terms of the resulting spectral sequence.

\subsection{$n$-cubes associated to torus action on a scheme}
\label{sect:SST-action}
Let us fix a basis $\{{\chi}_1, \cdots , {\chi}_r\}$ of the character 
group $T^{\vee} \simeq \Z^r$.
We define an action of $T$ on $\A^r_k$ by $t.x = y$, where 
$y_i = {\chi}_i (t)x_i$ for $1 \le i \le r$. Let $H_i$ be the coordinate
hyperplane in ${\A}^r_k$ given by the equation $(x_i = 0)$. 
Then each $H_i$ is a $T$-invariant closed subscheme of $\A^r_k$
and $T \simeq {\A}^r_k \setminus \stackrel{r}{\underset{i=1}{\cup}} H_i$.
For any element $I \in \<r\>$, let
\[
H_I =  \left\{ \begin{array}{ll}
{\underset{i \notin I}{\cap}} H_i & {\rm if} \ I \neq \{1, \cdots ,n\} \\
{\A}^r_k & {\rm if} \ I = \{1, \cdots , n\}.
\end{array}
\right.   
\]

For $I \subset J$, we have the natural inclusion $H_I \subset H_J$ and
$H_{\emptyset} = \stackrel{r} {\underset{i = 1}{\cap}} H_i$.
From this, we see that for a $T$-scheme $X$, 
the assignment $I \mapsto X \times H_I$ gives an $r$-cube $X_*$ in
the category of $T$-invariant closed subschemes of $X \times \A^r_k$
via the diagonal action of $T$ on the product. Let $r_I$ denote the
dimension of $H_I$. 

Let $j \ge 0$ be an integer and let $(V_j,U_j)$ be a good pair for $T$-action
corresponding to $j+r$. Set $\sZ^{(j)}_I = 
\sZ^{j+r_I}((X \times H_I) \stackrel{T}{\times} U_j, \bullet)$.
The following is then straightforward to check.

\begin{lem}\label{lem:n-cube}
The assignment $I \mapsto \sZ^{(j)}$ with inclusions sent to push-forward of
cycles defines an $n$-cube in $\Delta_{\Ab}$.
\end{lem}

We let ${\sZ}^{(j)}_*$ denote the $n$-cube considered in \lemref{lem:n-cube}
and let $\sZ_* = {\underset {j \ge 0}{\coprod}} {\sZ}^{(j)}_*$.

\begin{lem}\label{lem:fiberZ}
There is a natural isomorphism in $D^{-}(\Ab)$:
\[
{\rm fib} {\sZ}_* \xrightarrow{\simeq} {\underset{j \ge 0}{\coprod}}
{\sZ}^{j+r}({(X \times T)} 
\stackrel{T}{\times} U_j, \bullet) [-r].
\]
\end{lem}
\begin{proof}
We shall prove the lemma by induction on the length $r = |\<r\>|$ of 
cubes. For $r = 1$, this is simply the localization fiber sequence 
of Bloch \cite[Theorem~3.3]{Bloch}:
\[
{\sZ}^{j}({(X \times  \{0\})} \stackrel{T}{\times}U_j, \bullet)
\to 
{\sZ}^{j+1}({(X \times  {\A}^1_k)} \stackrel{T}{\times}U_j, \bullet)
\to 
{\sZ}^{j+1}({(X \times  T)} \stackrel{T}{\times}U_j, \bullet).
\]

Assuming $r \ge 2$ and the lemma holds for $r' \le r-1$,
we have
${\rm fib} {\sZ}^{(j)}_* = {\rm fib} ({\rm fib} {\sZ}^{(j)}_{*/r} \to {\rm fib} 
{\sZ}^{(j)}_{* \supset r})$ in the above notations.
Unwinding the definitions, one checks that ${\sZ}^{(j)}_{*/r}$ is canonically
identified with
the functor ${\sZ}^{(j)}$ applied to $H_r \simeq \A^{r-1}_k$ 
with closed subschemes $\{H_r \cap H_1, \cdots, H_r \cap H_{r-1}\}$. 
Setting $H_{r,l} = H_r \cap H_l$ for $1 \le l \le r-1$, it follows  
by induction that 
\begin{equation}\label{eqn:L0}
{\rm fib} {\sZ}^{(j)}_{*/r} \simeq
{\sZ}^{j+r-1}({(X \times (H_r \setminus 
\stackrel{r-1}{\underset{l = 1}{\cup}} H_{r,l}))} 
\stackrel{T}{\times} U_j, \bullet) [-r+1]
\end{equation}
in $D^{-}(\Ab)$.  

On the other hand, ${\sZ}^{(j)}_{* \supset r}$ is isomorphic to
the functor ${\sZ}^{(j)}_*$ applied to the $(r-1)$-cube of closed subschemes
of $\A^r_k$ given by
$X(I) = {\underset{1 \le l \le r-1}
{\underset{l \notin I}{\cap}}} H_l$ for $I \subset \{1, \cdots , r-1 \}$.
It follows again by induction that
\begin{equation}\label{eqn:L1}
{\rm fib} {\sZ}^{(j)}_{* \supset r} \simeq
{\sZ}^{j+r} ({(X \times (\A^r_k \setminus
\stackrel{r-1}{\underset{l = 1}{\cup}} H_l))}
\stackrel{T}{\times} U_j, \bullet) [-r+1].
\end{equation}
The lemma now follows by combining ~\eqref{eqn:L0}, ~\eqref{eqn:L1}, the 
localization fiber sequence
\[
{\sZ}^{j+r-1}({(X \times (H_r \setminus
\stackrel{r-1}{\underset{l = 1}{\cup}} H_{r,l}))} 
\stackrel{T}{\times} U_j, \bullet) \to
{\sZ}^{j+r} ({(X \times (\A^r_k \setminus 
\stackrel{r-1}{\underset{l = 1}{\cup}} H_l))} 
\stackrel{T}{\times} U_j, \bullet) \hspace*{2cm}
\]
\[
\hspace*{8cm} \to
{\sZ}^{j+r}({(X \times (\A^r_k \setminus 
\stackrel{r}{\underset{l = 1}{\cup}} H_l))} 
\stackrel{T}{\times} U_j, \bullet)
\]
and the identification 
$\A^r_k \setminus \stackrel{r}{\underset{l = 1}{\cup}} H_l = T$.
\end{proof}

Using Lemma~\ref{lem:fiberZ} and taking the homotopy groups of 
${\sZ}_I$ and ${\rm fib} {\sZ}_*$, the spectral sequence 
~\eqref{eqn:LSS} in our case becomes
\begin{equation}\label{eqn:LSS1*}
E^1_{p,q} = {\underset{|I| = r-p}{\oplus}}
\CH^*_T(X \times H_I, q) \Rightarrow
\CH^*_T(X \times T, p+q).
\end{equation} 

Since $T$ acts freely on itself and hence on $X \times T$ with quotient $X$,
it follows from \propref{prop:EHCG}(7) that the abutment of this spectral
sequence is isomorphic to $\CH^*(X, \bullet)$.
We therefore get a spectral sequence
\begin{equation}\label{eqn:LSS1}
E^1_{p,q} = {\underset{|I| = r-p}{\oplus}}
\CH^*_T(X \times H_I, q) \Rightarrow
\CH^*(X, p+q).
\end{equation}

To complete the proof of Theorem~\ref{thm:SS-Red-Main-I} for the torus $T$, 
we shall now identify the $E^1$-terms and the differentials of the spectral 
sequence ~\eqref{eqn:LSS1}. We shall use the following general
result for this purpose.

\begin{lem}\label{lem:moving0}
Let $G$ be a linear algebraic group acting on a smooth quasi-projective 
scheme $X$ over $k$.
Let $E \xrightarrow{p} X$ be a $G$-equivariant line bundle on $X$ and let  
$X \xrightarrow{f} E$ be the zero-section embedding. Then there exists
a pull-back map $\CH^*_G(E, \bullet) \xrightarrow{f^*}
\CH^*_G(X, \bullet)$ such that $f^* \circ p^* = id$ and
$f^* \circ f_*$ is the action of $c^G_1(E)$ on 
$\CH^*_G(X, \bullet)$.
\end{lem}
\begin{proof}
By choosing a good pair $(V, U)$ for the $G$-action and
considering the appropriate mixed quotients, we are reduced to proving
the non-equivariant version of the lemma.

Given a finite collection of irreducible closed subsets
$\sC$ of $E$ and integers $j, p \ge 0$, we let
$\sZ^j_{\sC}(X,p)$ be the subgroup of $\sZ^j(X,p)$ generated by integral
closed subschemes $Z \subset X \times \Delta^p$ such that for each $C \in \sC$
and each face $F \subset \Delta^p$, we have
$\dim(Z \cap (C \times F)) \le \dim(C) + \dim(F)- j$.
It is straightforward to check that $\{\sZ^j_{\sC}(X,p)\}_{p \ge 0}$
together form a subcomplex of $\sZ^j(X,\bullet)$
(see, for instance, \cite[Definition~1.9]{KL}).
One knows that a combination of the localization sequence for the
cycle complexes of all quasi-projective schemes (possibly singular)
from \cite{Bloch-1} and the moving lemma for smooth projective 
schemes from \cite{Bloch} proves that the inclusion
$\sZ^j_{\sC}(E, \bullet) \inj \sZ^j(X,\bullet)$ is a quasi-isomorphism.
In our case, we conclude that the inclusion
${\sZ}^j_X(E, \bullet) \to {\sZ}^j(E, \bullet)$ is a quasi-isomorphism.

On the other hand,
there is a natural map ${\sZ}^j_X(E, \bullet) \xrightarrow{f^*}
{\sZ}^j(X, \bullet)$ given by intersecting any cycle with $X$.
From this, it is clear that the induced map $f^* \circ p^*$ on 
$\CH^*(X, \bullet)$ is identity.
Moreover, since $X$ is a Cartier divisor on $E$ and since $f^*$ is
defined by intersection with $X$, one has $f_* \circ f^* = c_1(L)$,
where $L$ is the line bundle on $E$ defined by the zero-section.
In particular, we have for any $x \in \CH^*(X, \bullet)$,
\[
\begin{array}{lll}
f^* \circ f_*(x) & = & f^* \circ f_*(f^* \circ p^*(x)) \\
& = & f^*(c_1(L) \cdot p^*(x)) \\ 
& = & f^*(c_1(L)) \cdot (f^* \circ p^*(x)) \\
& = & c_1(f^*(L)) \cdot x \\
& = & c_1(N_{X/E}) \cdot x \\
& = & c_1(E) \cdot x
\end{array}
\]   
and this proves the lemma.
\end{proof}

\subsection{The equivariant to ordinary Chow groups spectral
sequence}\label{sect:EOHCGSS}
We now identify the $E^1$-terms and differentials of the spectral sequence 
~\eqref{eqn:LSS1} with the help of ~\eqref{eqn:LSS0} and \lemref{lem:moving0}.
For a fixed $I \in \<r\>$ with $|I| = r-p$ and 
$j \notin I$, let $J = I \cup \{j\}$. We can then identify $H_J$ with
$H_I \times {\A}^1_k$, where the action of $T$ on this product is induced
by the action of $T$ on $\A^r_k$ as described above.
Identifying $\CH^*_T\left(X \times H_I, q\right)$ and
$\CH^*_T\left(X \times H_J, q\right)$ with 
$\CH^*_T\left(X, q\right)$ via the homotopy invariance (see 
\propref{prop:EHCG}(2)), the differential
of ~\eqref{eqn:LSS1} is given by the composite
\[
\CH^*_T\left(X, q\right) \xrightarrow{\iota_*} 
\CH^*_T\left(X \times {\A}^1_k, q\right) \xrightarrow{\iota^*}
\CH^*_T\left(X, q\right),
\] 
where $\iota$ is the zero-section embedding.

It follows from 
Lemma~\ref{lem:moving0} that this composite map is multiplication by
$c^T_1(N)$, where $N$ is the equivariant normal bundle of the origin in 
the one-dimensional representation ${\A}^1_k$ of $T$.
On the other hand,
this representation is given by a generator $\chi$ of $\G_m^{\vee}$ with 
corresponding line bundle $L_{\chi}$. 
As the class of $\chi$ in $K^T_0(k)$ is given by 
${\chi} - 1$, we get $c^T_1(N) = c^T_1({\chi} - 1) = L_{\chi}$.
We conclude from this and the description of the differentials in
~\eqref{eqn:LSS0} that 
\[
d^1_{p,q}: {\underset{|I| = r-p}{\oplus}}
\CH^*_T(X \times H_I, q)  \to
{\underset{|J| = r-p+1}{\oplus}}
\CH^*_T(X \times H_J, q) 
\]
is of the form
\begin{equation}\label{eqn:DL}
d^1_{p,q}(xe_I) = \stackrel{p}{\underset{l= 1}{\sum}}
{(-1)}^{l-1} (L_{{\chi}_{i_l}} x) e_{I \cup i_l},
\end{equation}
where $\{i_1, \cdots , i_p\} = I^c = \{1, \cdots , r\} \setminus I$.
Here, $xe_I$ denotes an element $x \in \CH^*_T(X, q)$
under the identification $\CH^*_T(X, q) \xrightarrow{\simeq}
\CH^*_T(X \times H_I, q)$.

Let $S = S(T)$ denote the ring $\CH^*_T(k,0)$. One knows that
this is isomorphic to the polynomial ring ${\Z}[t_1, \cdots , t_r]$,
where $t_l$ is the cycle class of the line bundle $L_{{\chi}_l}$
(see \cite[\S~3.2]{ED2}).
Let $K_{\bullet} = \stackrel{r}{\underset{l = 1}{\otimes}}
K_{\bullet}(S \xrightarrow{t_l} S)$ denote the Koszul
resolution of $S/{J_T}$.

\begin{lem}\label{lem:Koszul}
There is a canonical isomorphism of the complexes of $S$-modules
\[
E^1_{\bullet, q} \xrightarrow{\simeq} K_{\bullet} {\otimes}_S 
\CH^*_T\left(X, q \right),
\]
where $E^1_{\bullet, q}$  is the spectral sequence ~\eqref{eqn:LSS1}.
\end{lem}
\begin{proof} 
The terms of $K_{\bullet}$ are given by 
$K_p = {\underset{1 \le i_1 < \cdots < i_p \le r}{\bigoplus}}
S {e_{i_1} \cdots e_{i_p}}$, which is a free $S$-module of rank
$p \choose r$ with basis
$\{f_I = e_{i_1 \cdots e_{i_p}}| 1 \le i_1 < \cdots < i_p \le r\}$ for 
$1\le p \le r$ and $K_0 = S$.  
The differentials of this complex are given by
\[
d(e_{i_1 \cdots e_{i_p}}) =
\stackrel{p}{\underset{l= 1}{\sum}}
{(-1)}^{l-1} t_{i_l} e_{i_1 \cdots {\widehat{i_l}} \cdots e_{i_p}}
\]
for $1 \le p \le r$ and $d(e_l) = t_l$ for $p = 1$.
Comparing this with ~\eqref{eqn:DL}, it is now easy to see that the map 
${\phi}: E^1_{\bullet, q} \to
K_{\bullet} {\otimes}_S \CH^*_T\left(X, q \right)$, given by
${\phi}\left(xe_I\right) = xf_{I^c}$,
yields the desired isomorphism of the two complexes.
\end{proof}

We now prove the torus part of \thmref{thm:SS-Red-Main-I}, restated
as follows.

\begin{thm}\label{thm:Sseq-T}
Let $T$ be a split torus acting on a smooth quasi-projective 
scheme $X$ over $k$. 
Then there is a convergent homological spectral sequence 
\begin{equation}\label{eqn:SS0-T} 
E^2_{p,q} = {\rm Tor}^{S(T)}_p ({\Z}, \CH^*_T(X, q))
\Rightarrow \CH^*(X, p+q)
\end{equation}
such that the edge homomorphism $\CH^*_T(X,i)
{\otimes}_{S(T)} {\Z} \to \CH^*(X, i)$
is induced by the forgetful map $\CH^*_T(X, i) \to
\CH^*(X, i)$. 
\end{thm}
\begin{proof}
We have seen above that $K_{\bullet}$ is
a free $S$-module resolution of $\Z = S/{J_T}$. Hence, it follows 
at once from Lemma~\ref{lem:Koszul} that there is a convergent
spectral sequence as in ~\eqref{eqn:SS0-T} with the desired $E^2$-terms.
Moreover, ~\eqref{eqn:LSS1} shows that this spectral sequence 
converges to $\CH^*(X, p+q)$.
\end{proof}

\begin{exm}\label{exm:SS*}
We give an example to demonstrate how the spectral sequence of
Theorem~\ref{thm:Sseq-T} works. We first remark that the homological
nature of this spectral sequence implies 
(see \cite[\S~5.2]{Weibel1}) that the differential of this 
spectral sequence
is given by $d^r : E^r_{p,q} \to E^r_{p-r, q+r-1}$.
Moreover, there is a filtration
\[
0 = F_{-1} \CH^*(X,n) \subseteq \cdots \subseteq F_{p-1}
\CH^*(X,n) \subseteq 
F_p \CH^*(X,n) \subseteq 
\]
\[
\hspace*{9cm} \cdots \subseteq F_t \CH^*(X,n) 
= \CH^*(X,n) 
\]
such that $E^{\infty}_{p,q} = {F_p}/{F_{p-1}}$.

Now consider $T = {\G}_m$ acting on itself
by left multiplication which is the restriction of the linear action of
$\G_m$ on $\A^1_k$ with weight $1$. In particular, we have $S(T) \simeq
\Z[t]$, the polynomial ring.
It follows from \propref{prop:EHCG}(4) that there is a split exact
sequence
\begin{equation}\label{eqn:SS*1}
0 \to \CH^j_{\G_m}(k,i) \to  \CH^j_{\G_m}(\G_m,i) \to  \CH^{j-1}_{\G_m}(k,i-1)
\to 0.
\end{equation}
We have a similar split exact sequence of ordinary higher Chow groups. 

Since $B\G_m \simeq {\P}^{\infty}_k$, the projective bundle formula for the 
ordinary Chow group yields
\[
\CH^*_{\G_m}(k,1) \cong \Z[t] {\otimes}_{\Z} \CH^*(k,1) \ {\rm and}
\]
\[
\CH^*_{\G_m}(k,0) \cong \Z[t] {\otimes}_{\Z} \CH^*(k,0).
\]
Combining these identifications with ~\eqref{eqn:SS*1}, we get
\begin{equation}\label{eqn:SS*3}
\CH^*_{\G_m}(\G_m,1) 
\cong (k^* \otimes_{\Z} \Z[t]) \oplus \Z[t][-1],
\ \CH^*_{\G_m}(\G_m,0) \cong \Z[t],  
\end{equation}
\[
\CH^*(\G_m,1) =  k^* \oplus \Z, \ {\rm and} \
\CH^*(\G_m,0)
= \Z.
\]
In particular, we have 
\begin{equation}\label{eqn:SS*4}
\Z {\otimes}_{\Z[t]} \CH^*_{\G_m}(\G_m,1) \xrightarrow{\simeq} \CH^*(\G_m,1).
\end{equation}

Now we compute $\CH^*(\G_m,1)$ using Theorem~\ref{thm:Sseq-T}.
It is easy to see that
\begin{equation}\label{eqn:E-2-pq}
E^2_{p,q} = \left\{\begin{array}{ll}
{{\Z} {\otimes}_{\Z[t]} \CH^*_{\G_m}(\G_m,q)} & \mbox{if \ $p =0$} \\
{ {(\CH^*_{\G_m}(\G_m,q))}_{\rm tor}} & \mbox{if \ $p=1$} \\
0 & \mbox{otherwise .}
\end{array}
\right.
\end{equation}
This gives an exact sequence
\[
0 \to \Z {\otimes}_{\Z[t]} \CH^*_{\G_m}(\G_m,1) \to \CH^*(\G_m,1)
\to {\CH^*_{\G_m}(\G_m,0))}_{\rm tor} \to 0.
\]
Since the right end term is zero by ~\eqref{eqn:SS*3}, we see that this yields
the same answer as in ~\eqref{eqn:SS*4} (but in a much shorter way).
\end{exm}

\section{Proof of Theorem~\ref{thm:SS-Red-Main-I}: 
the general case}\label{sect:SS-R}
In this section, we shall use \thmref{thm:Sseq-T} to complete the proof of
Theorem~\ref{thm:SS-Red-Main-I}.
We first prove some intermediate results about
the relation between the equivariant Chow groups for the action of
a split reductive group and its maximal tori.
These results are of independent interest.
Let us fix a reductive group $G$, a split maximal torus $T \subset G$ of rank 
$r$ and a Borel subgroup $B \subset G$ containing $T$.
Recall from \S~\ref{sect:MR*} that $t_G$ denotes the
torsion index of $G$. One knows that
$t_G$ is the order of the cokernel of the map 
$r_{G/B}: S(T)_N \to \CH^N(G/B,0) \simeq \Z$, where $N = \dim(G/B)$.

Let $R$ denote the ring
$\Z[t^{-1}_G]$. Since we wish to work with $R$-modules in this section,
we shall assume all abelian groups in this section to be tensored with
$R$ (over $\Z$).  

\subsection{Reduction to the case of split tori}\label{sect:Red-T}
Recall from ~\eqref{eqn:res}  that for $X \in \Sch^G_k$, 
there is a restriction
map $\CH^*_G(X, \bullet) \to \CH^*_T(X, \bullet)$ which induces a map
of $S(T)$-modules
\begin{equation}\label{eqn:Res-0}
\lambda_X : \CH^*_G(X, \bullet) \otimes_{S(G)} S(T) \to \CH^*_T(X, \bullet).
\end{equation}
This is an $S(T)$-algebra homomorphism if $X$ is smooth.

The forgetful map from the equivariant to the ordinary Chow groups yields an 
$S(G)$-linear map $r_{G/B}: \CH^*_G(G/B) \otimes_{S(G)} R \to \CH^*(G/B)$.
Identifying $\CH^*_G(G/B)$ with $S(T) \simeq S(B)$ 
(see \cite[Corollary~3.2, Proposition~3.3]{Krishna1}),
we get the so called {\sl characteristic} ring homomorphism
\[
r_{G/B}: S(T) \otimes_{S(G)} R \to \CH^*(G/B).
\]

\begin{lem}\label{lem:Char}
The map $r_{G/B}$ is an isomorphism.
\end{lem}
\begin{proof}
This is proven for cobordism in \cite[Corollary~5.2]{KK}
and a similar (also simpler) proof also works for Chow groups.
We present here another (but related) proof.
Let $W = N(T)/T$ denote the Weyl group of $G$. Then
$W$ acts naturally on $S(T) \simeq \CH^*_T(k)$ as $S(G)$-algebra
homomorphisms. Let $S(T)^W$ denote the ring of invariants.
It follows from \cite[Theorem~10.2]{CZ} (and the comments following
the theorem) that the map
$r^T_{G/B}:S(T) \otimes_{{S(T)}^W} S(T) \to \CH^*_T(G/B)$ is an $S(T)$-algebra
isomorphism. On the other hand, it follows from
\cite[Theorem~5.7]{Krishna1} that the map
$S(G) \to {S(T)}^W$ is an isomorphism. 
We conclude that the characteristic map
\begin{equation}\label{eqn:G-Torus}
r^T_{G/B} : S(T) \otimes_{S(G)} S(T) \simeq S(T) \otimes_{S(G)} \CH^*_G(G/B) 
\to \CH^*_T(G/B)
\end{equation}
is an isomorphism.

Since $G/B$ is smooth and projective over $k$, it follows immediately from
\thmref{thm:Sseq-T} that the map $\CH^*_T(G/B) \otimes_{S(T)} R \to \CH^*(G/B)$
is an isomorphism. In particular, we get
\[
S(T) \otimes_{S(G)} R \simeq 
(S(T) \otimes_{S(G)} S(T)) \otimes_{S(T)} R \simeq
\CH^*_T(G/B) \otimes_{S(T)} R \simeq \CH^*(G/B)
\]
and this proves the lemma.
\end{proof}

We wish to prove the following refinement of \thmref{thm:reductiveI}. 
Analogous result for the equivariant $K$-theory
of $GL_n$-action (more generally, groups whose commutator subgroup
is simply connected) was proven by Merkurjev \cite[Proposition~8]{Merkurjev}.

\begin{thm}\label{thm:Sspl-M}
For any $X \in \Sch^G_k$, the map $\lambda_X$ of ~\eqref{eqn:Res-0}
is an isomorphism.
\end{thm}
\begin{proof}
We shall closely follow the proof of \cite[Theorem~8.5]{Krishna1}
but will supply the necessary details.
As shown in {\sl ibid.}, it suffices to show that for any $k$-scheme $Y$
and a $G$-torsor $f: X \to Y$, the natural map
\begin{equation}\label{eqn:Sspl-M-0}
\lambda_{Y} : \CH^*(Y, \bullet) {\otimes}_{S(G)} S(T) \to \CH^*(X/B, \bullet)
\end{equation}
is an isomorphism. 
It suffices to show that $\lambda_Y$ is an isomorphism after localizing
at every prime $p$ not dividing $t_G$. So we fix such a prime $p$.
We shall prove the theorem by noetherian induction on 
$Y$, which we can assume to be reduced.
For any map $Y' \to Y$ in $\Sch_k$, we let $X_{Y'} = X \times_Y Y'$.

If $f$ is the trivial $G$-torsor so that 
$G/B \times Y \xrightarrow{f} Y$ is the trivial flag bundle, we get
\[
\begin{array}{lll}
\CH^*(Y, \bullet) \otimes_{S(G)} S(T) & \simeq & \CH^*(Y, \bullet) 
\otimes_{\CH^*(k)} 
\left(\CH^*(k) \otimes_{S(G)} S(T)\right) \\
& {\simeq}^1 & \CH^*(Y, \bullet) \otimes_{\CH^*(k)} \CH^*(G/B) \\
& {\simeq}^2 & \CH^*(Y \times G/B, \bullet),
\end{array}
\]
where ${\simeq}^1$ follows from \lemref{lem:Char} and
${\simeq}^2$ follows from \cite[Lemma~6.2]{Krishna1}.

In general, it follows from \cite[Th\'eor\`eme~2]{Groth} that every $G$-torsor
over any given scheme point of $Y$ becomes trivial after a finite
separable field 
extension of degree prime to $p$. We remark here that this result was
proven by Grothendieck only for the geometric points. It was later generalized 
to the case of all points of $Y$ by Totaro \cite[Theorem~1.1]{Totaro-1}.
We conclude from this that there is a non-empty open subset $j: U \inj Y$
and a finite {\'e}tale cover $g: U' \to U$ of degree prime to $p$
such that $f$ becomes trivial over $U'$.

We have a commutative diagram of
the flat pull-back and proper push-forward maps
\begin{equation}\label{eqn:Sspl-M-1}
\xymatrix@C1pc{
\CH^*(U, \bullet) \otimes_{S(G)} S(T) \ar[r]^-{\lambda_U} \ar[d]_{g^*} &
\CH^*({X_U}/B, \bullet) \ar[d]^{g^*} \\
\CH^*(U', \bullet) \otimes_{S(G)} S(T) \ar[r]^-{\lambda_{U'}} \ar[d]_{g_*} &
\CH^*({X_{U'}}/B, \bullet) \ar[d]^{g_*} \\
\CH^*(U, \bullet) \otimes_{S(G)} S(T) \ar[r]^-{\lambda_U}  &
\CH^*({X_U}/B, \bullet)}
\end{equation}
such that the composite vertical arrows are multiplication by
an integer not divisible by $p$. Since $\lambda_{U'}$ is an
isomorphism after localization at $p$, we conclude that the same holds for
$\lambda_U$ too.

We let $i: Z \inj Y$ be the complement of $U$ and consider the commutative
diagram of localization exact sequences
\begin{equation}\label{eqn:FL-Chow-grps&}
\xymatrix@C.5pc{
{\begin{array}{c}
\CH^*(U, \bullet) \\
{\otimes} \\
S(T)
\end{array}} 
\ar[r]^{\partial} \ar[d]_>>>>>{\lambda_U} & 
{\begin{array}{c}
\CH^*(Z, \bullet) \\
{\otimes} \\
S(T)
\end{array}}
\ar[d]^>>>>>{\lambda_Z} \ar[r]^{\iota_*} &
{\begin{array}{c}
\CH^*(Y, \bullet) \\
{\otimes} \\
S(T)
\end{array}} 
\ar[d]^>>>>>{\lambda_X} \ar[r]^{j^*} &
{\begin{array}{c}
\CH^*(U, \bullet) \\
{\otimes} \\
S(T)
\end{array}} 
\ar[r]^{\partial} \ar[d]^>>>>>{\lambda_U} & 
{\begin{array}{c}
\CH^*(Z, \bullet) \\
{\otimes} \\
S(T)
\end{array}}
\ar[d]^>>>>>{\lambda_Z} \\
\CH^*({X_U}/B, \bullet) \ar[r] & \CH^*({X_Z}/B, \bullet) \ar[r] &
\CH^*(X/B, \bullet) \ar[r] & \CH^*({X_U}/B, \bullet) \ar[r] & 
\CH^*({X_Z}/B, \bullet),}
\end{equation}
where the tensor product in the top row is over the ring $S(G)$.
It follows from \cite[Theorem~5.7]{Krishna1} that the canonical map
$S(G) \to {S(T)}^W$ is a ring isomorphism. 
On the other hand, \cite[Th\'eor\`eme~2]{Dem1} says that
$S(T)$ is a free ${S(T)}^W$-module of finite rank.
It follows that the top row is exact. The second and the third squares 
commute by the compatibility of $\lambda_Y$ with the push-forward and the 
pull-back maps. 

To see that
the first square commutes, let us consider an element $a \otimes b
\in \CH_*(U, \bullet) {\otimes} S(T)$. If we map this horizontally, we get
$\partial(a) \otimes b$ which is mapped vertically down to 
$b \cdot f^*_Z \circ \partial(a)$. Since the localization sequence
of higher Chow groups is compatible with respect to the flat pull-back,
this last term is same as $b \cdot \partial \circ f^*_U (a)$.
On the other hand, mapping $a \otimes b$ vertically down gives $b \cdot    
f^*_U(a)$ and if we map this horizontally, we get 
$\partial\left(b \cdot f^*_U(a)\right)$. Since the horizontal maps
in the bottom row are $S(T)$-linear (see \cite[Proposition~2.2]{Krishna1}), 
we conclude that the first (hence the fourth) square commutes. 

We have shown above that (after localization at $p$)
the first and the fourth vertical arrows 
in ~\eqref{eqn:FL-Chow-grps&} are isomorphisms. 
The second and the fifth vertical arrows are isomorphisms by 
noetherian induction. We conclude from 5-lemma that $\lambda_Y$ is an 
isomorphism. 
\end{proof}

\vskip .3cm

{\sl Proof of \thmref{thm:SS-Red-Main-I}:}
Let $X \in \Sm^G_k$ and let $C_{\bullet} \surj \CH^*_G(X, q)$ be a free 
$S(G)$-module resolution of $\CH^*_G(X, q)$ for an integer $q \ge 0$. 
We saw during the proof of \thmref{thm:Sspl-M} that $S(T)$ is
a free $S(G)$-module of finite rank.
We can thus apply \thmref{thm:Sspl-M} to conclude that $C_{\bullet} 
{\otimes}_{S(G)} S(T)$ is a free $S(T)$-module resolution of
$\CH^*_T(X, q)$. This in turn implies that
\begin{equation}\label{eqn:General}
{\rm Tor}^{S(G)}_p(R, \CH^*_G(X, q)) 
\xrightarrow{\simeq} {\rm Tor}^{S(T)}_p(R, \CH^*_T(X, q))
\end{equation}
for all $p, q \ge 0$. We conclude the proof by \thmref{thm:Sseq-T}.
$\hfill\square$

\subsection{Applications of \thmref{thm:SS-Red-Main-I}:}
\label{sect:Applns}
We conclude this section with some important applications of 
\thmref{thm:SS-Red-Main-I}. Our first application is the following
explicit relation between the equivariant and the ordinary Chow groups
of schemes with group action. Its significance comes from the fact that
the equivariant Chow groups are often simpler objects to deal with
and there are various techniques to compute them (e.g., see \cite{Brion}).
So the result below allows us to compute (the more difficult) ordinary Chow
groups of such schemes. When $G$ is a split torus, this recovers 
\cite[Corollary~2.3]{Brion}. The proof is immediate from the
spectral sequence ~\eqref{eqn:SS-M-01}.

\begin{cor}\label{cor:Chow0}
Let $G$ be a split reductive group acting on a smooth 
quasi-projective scheme $X$ over $k$. Then the forgetful map
\[
\CH^*_G(X) \otimes_{S(G)} R \to \CH^*(X)
\]
is an isomorphism of $R$-modules. 
This is a ring isomorphism if $X$ is smooth over $k$.
\end{cor}

The second application of Theorem~\ref{thm:SS-Red-Main-I} is the following 
explicit relation between equivariant and 
ordinary higher Chow groups of schemes with $\G_m$-action.
The proof follows easily using the computations in
~\eqref{eqn:E-2-pq}.
\begin{cor}\label{cor:G-m-action}
For a smooth quasi-projective scheme $X$ over $k$ with $\G_m$-action and 
an integer $n \ge 0$, there is a natural 
exact sequence
\[
0 \to \Z {\otimes}_{\Z[t]} \CH^*_{\G_m}(X,n) \to \CH^*(X,n) \to
{(\CH^*_{\G_m}(X,n-1))}_{\rm tor} \to 0.
\]
\end{cor}

\subsubsection{Higher Chow groups of principal and flag bundles}
\label{sect:P-bundles}
Let $X$ be a smooth quasi-projective $k$-scheme and let $f: E \to X$ be a torsor
for a split reductive group $G$. Let $B$ be a Borel subgroup of $G$ 
containing a split maximal torus $T$. An important result of Vistoli 
\cite[Theorem~3.1]{Vistoli} in the intersection theory of principal 
and flag bundles says that the pull-back map $f^*$ induces an isomorphism
\[
\CH^*(X)_{\Q} \otimes_{S(G)_{\Q}} S(T)_{\Q} \xrightarrow{\simeq}
\CH^*(E/B)_{\Q}.
\]
Using this, Vistoli \cite[Corollary~3.2]{Vistoli} obtains an isomorphism
\[
\CH^*(X)_{\Q} \otimes_{S(G)_{\Q}} \Q \xrightarrow{\simeq} \CH^*(E)_{\Q}.
\]

As our next application, we obtain the following full generalizations and 
refinements of above two results. 
The part (1) of the result below follows at once from \thmref{thm:Sspl-M}
and \propref{prop:EHCG}(7). The part (2) follows immediately from
\thmref{thm:SS-Red-Main-I}, \corref{cor:Chow0} and \propref{prop:EHCG}(7).

\begin{thm}\label{thm:G-torsor}
Given a $G$-torsor $f: E \to X$ as above, the following hold.
\begin{enumerate}
\item
For every $q \ge 0$, the natural map
\[\CH^*(X,q)
{\otimes}_{S(G)} {S(T)} \to \CH^*(E/B, q)
\]
is an isomorphism.
\item
There is a convergent spectral sequence
\[
E^2_{p,q} = {\rm Tor}^{S(G)}_p (R, \CH^*(X, q))
\Rightarrow \CH^*(E, p+q).
\]
The edge homomorphism induces an isomorphism
\[
\CH^*(X)
{\otimes}_{S(G)} R \xrightarrow{\simeq} \CH^*(E).
\]
\end{enumerate}
\end{thm}

\section{Degeneration of spectral sequence for projective 
schemes}\label{sect:Proj}
In this section, we shall prove the degeneration of the spectral
sequence ~\eqref{eqn:SS-M-01} for smooth projective schemes.
Our main tool to prove this is  
the following stratification theorem  for the action of a split torus.

\begin{thm}[Bialynicki-Birula, Hesselink, Iverson]\label{thm:BBH}
Let $T$ be a split torus acting on a smooth projective scheme $X$ over a
field $k$. Then:
\begin{enumerate}
\item
The fixed point locus $X^T$ is a smooth closed subscheme of X. 
\item
There is an ordering $X^T = \stackrel{n}{\underset{j=0}{\coprod}}
Z_j$ of the connected components of $X^T$, a 
filtration of $X$ by $T$-invariant closed subschemes
\[
{\emptyset} = X_{-1} \subset X_0 \subset \cdots \subset X_n = X
\]
and maps ${\phi}_j : X_j \setminus X_{j-1} \to Z_j$ for $0 \le j \le n$ which
are all $T$-equivariant affine bundles. 
\end{enumerate}
\end{thm}  

The part (1) of the above theorem was proven by Iverson \cite{Iverson}
when $k$ is algebraically closed. Both parts of the theorems were proven
by Bialynicki-Birula \cite{BB} when $k$ is algebraically closed.
In this case, Bialynicki-Birula \cite{BB} actually showed that
maps ${\phi}_j : X_j \setminus X_{j-1} \to Z_j$ for $0 \le j \le n$ which
are all $T$-equivariant vector bundles.
The most general case of the theorem was proven by Hesselink \cite{Hessel}.

\begin{prop}\label{prop:GeneralSplit}
Let $T$ be a split torus acting on a smooth projective scheme $X$ over $k$. 
Let $Z_j$'s be the connected components of $X^T$ as in Theorem~\ref{thm:BBH}. 
Then for every $q \ge 0$, there is an isomorphism of $S(T)$-modules
\[
\stackrel{n}{\underset{j=0}{\oplus}} \CH^*_T(Z_j, q)
\xrightarrow{\simeq} \CH^*_T(X, q).
\]
\end{prop} 
\begin{proof} 
We prove the proposition by induction on the length of the filtration.
For $n = 0$, the inclusion $Z_0 \inj X_0$ is the
0-section embedding of the $T$-equivariant affine bundle
$X = X_0 \xrightarrow{{\phi}_0} Z_0$. Hence, the proposition follows from the
homotopy invariance (see Remark~\ref{remk:shift}).

We now assume by induction that $1 \le m \le n$ and 
we have an $S(T)$-module decomposition 
\begin{equation}\label{eqn:split0}
\stackrel{m-1}{\underset{j=0}{\oplus}} \CH^*_T(Z_j,q)\xrightarrow{\simeq} 
\CH^*_T(X_{m-1},q).
\end{equation}

The localization exact sequence for the 
inclusions $i_{m-1} : X_{m-1} \inj X_m$ and $j_{m} : W_{m} = 
X_m \setminus X_{m-1}\inj X_m$ 
of the $T$-invariant closed and open 
subschemes yields a long exact sequence of $S(T)$-linear maps

\begin{equation}\label{eqn:split*1}
\cdots \to \CH^*_T(X_{m-1},q) \xrightarrow{i_{(m-1)*}}
\CH^*_T(X_m,q) \xrightarrow{j^*_m}
\CH^*_T(W_m,q) \xrightarrow{\partial} 
\CH^*_T(X_{m-1}, q-1) \to \cdots.
\end{equation}
Using ~\eqref{eqn:split0}, it suffices now to construct an 
$S(T)$-linear splitting 
of the pull-back $j^*_m$ in order to prove the proposition.  

Let $V_m \subset W_m \times Z_m$ be the graph of the projection 
$W_m \xrightarrow{{\phi}_m} Z_m$ and let $Y_m$ 
denote the closure of $V_m$ in $X_m \times Z_m$.  
Then $Y_m$ is a $T$-invariant closed subset of $X_m \times Z_m$ and $V_m$
is $T$-invariant and open in $Y_m$. 
We consider the composite maps 
\begin{equation}\label{eqn:split01}
p_m : V_m \inj W_m \times Z_m \to W_m, \ \ 
q_m : V_m \inj W_m \times Z_m \to Z_m \ \ {\rm and} 
\end{equation}
\[
{\ov{p}}_m : Y_m \inj X_m \times Z_m \to X_m, \ \ 
{\ov{q}}_m : Y_m \to X_m \times Z_m \to Z_m 
\]
in $\Sch^T_k$. Note that ${\ov{p}}_m$ is a projective morphism since $Z_m$ is 
projective. The map $q_m$ is smooth and $p_m$ is an isomorphism. 

We consider the diagram 
\begin{equation}\label{eqn:split1}
\xymatrix{
{\CH^*_T\left(Z_m,q\right)} \ar[r]^{{\ov{q}}^*_m} 
\ar[d]_{{\phi}^*_m}^{\simeq} &
{\CH^*_T\left(Y_m,q\right)} \ar[d]^{{{\ov{p}}_m}_*} \\
{\CH^*_T\left(W_m,q\right)} & 
{\CH^*_T\left(X_m,q\right)} \ar[l]^{j^*_m}.}
\end{equation} 

The map ${{\phi}^*_m}$ is an isomorphism by the homotopy invariance. 
We also observe that the map ${\ov{q}}^*_m$ exists since $Z_m$ is smooth
(see \propref{prop:EHCG}(1)).
It suffices to show that this diagram commutes. For, the map 
$s_m : = {{\ov{p}}_m}_* \circ {\ov{q}}^*_m \circ
{{\phi}^*_m}^{-1}$ will then give the desired splitting of the map
$j^*_m$. Note that $s_m$ is $S(T)$-linear since so are all the maps in
~\eqref{eqn:split1}.

We now consider the commutative diagram in $\Sch^T_k$:
\[
\xymatrix@C3pc{
X_m & W_m \ar[l]_{j_m}& \\
Y_m \ar[u]^{{\ov{p}}_m} \ar[dr]_{{\ov{q}}_m} & V_m \ar[u]_{p_m} \ar[d]^{q_m}
\ar[l]_{{\ov{j}}_m} & W_m \ar[ul]_{id} 
\ar[l]^{(id, {\phi}_m)} \ar[dl]^{{\phi}_m} \\
& Z_m. & }
\]
Since the top left square is Cartesian with ${\ov{p}}_m$ projective and 
$j_m$ an open immersion, it follows that $j^*_m \circ {{\ov{p}}_m}_*  = 
{p_m}_* \circ {\ov{j}}^*_m$. 
Now, using the fact that $(id, {\phi}_m)$ is an isomorphism, 
we get 
\[
\begin{array}{lllll}
j^*_m \circ {{\ov{p}}_m}_* \circ {\ov{q}}^*_m & = & 
{p_m}_* \circ {\ov{j}}^*_m \circ {\ov{q}}^*_m & = &  
{p_m}_* \circ q^*_m \\
& = & {p_m}_* \circ {(id, {\phi}_m)}_* \circ {(id, {\phi}_m)}^* \circ
q^*_m & = & {id}_* \circ {\phi}^*_m \\
& = &  {\phi}^*_m. & &  
\end{array} 
\]
This proves the commutativity of ~\eqref{eqn:split1} and hence the 
proposition.
\end{proof}

\begin{remk}\label{remk:GeneralSplit*}
We remark that the non-equivariant version of \propref{prop:GeneralSplit}
is known to exist in the literature in various forms. 
For instance, in view of the comparison of 
Chow motives and Voevodsky's derived category of motives and the 
relation between the motivic cohomology and higher Chow groups,
the non-equivariant version follows from the results of
Brosnan \cite{Brosnan}. The results of Brosnan build on the earlier work of
Karpenko \cite{Karpenko}, generalized by Chernousov-Gille-Merkurjev
\cite{CGM} (see also \cite{DB}).
\end{remk}

Using \propref{prop:GeneralSplit}, we  obtain the following explicit
relation between equivariant and ordinary higher Chow groups of
smooth projective schemes with torus action. The case $q = 0$ of this
result was earlier 
proven by Brion (see \cite[Corollaries~2.3 and 3.1]{Brion}). 

\begin{thm}\label{thm:split} Let $T$ be a split torus acting on a 
smooth projective scheme $X$ over $k$. Then for every $q \ge 0$, there is  
an isomorphism of $S(T)$-modules
\begin{equation}\label{eqn:DSS0}
\CH^*(X, q) {\otimes}_{\Z} S(T) 
\xrightarrow{\simeq} \CH^*_T(X, q).
\end{equation}
\end{thm}
\begin{proof}
Since $T$ acts trivially on each $Z_j$, it follows from
Proposition~\ref{prop:GeneralSplit} and \cite[Theorem~2.4]{KrishnaV} that
there are $S(T)$-module isomorphisms
\[
({\stackrel{n}{\underset{j=0}{\oplus}} \CH^*(Z_j, q)}) 
{\otimes}_{\Z} S(T) \simeq \ \
\stackrel{n}{\underset{j=0}{\oplus}} (\CH^*(Z_j, q)
{\otimes}_{\Z} S(T))
\xrightarrow{\simeq} 
\stackrel{n}{\underset{j=0}{\oplus}} \CH^*_T(Z_j, q)
\xrightarrow{\simeq} \CH^*_T(X, q).
\] 
On the other hand, if we apply Proposition~\ref{prop:GeneralSplit}
by taking $T = \{1\}$, we see that 
the top left term is same as $\CH^*(X, q)
{\otimes}_{\Z} S(T)$. This yields ~\eqref{eqn:DSS0}.
\end{proof}

The following generalization of \cite[Theorem~3.2]{Brion} to the 
equivariant higher Chow groups is immediate from Theorem~\ref{thm:split}. 

\begin{cor}\label{cor:split*} 
Let $T$ be a split torus acting on a smooth projective scheme $X$ over $k$. 
Then for every $q \ge 0$, the ${S(T)}_{\Q}$-module 
${\CH^*_T(X, q)}_{\Q}$ is free.
\end{cor}

As an application of \thmref{thm:split}, we now prove the following
main result of this section. Recall from \S~\ref{sect:SS-R}
that $t_G$ denotes the torsion index of a split reductive
group $G$.

\begin{thm}\label{thm:Degn}
Let $G$ be a split reductive group acting on a smooth projective
scheme $X$ over $k$ and let $q \ge 0$ be any integer.
Then the natural map
\[
\CH^*_G(X,q) \otimes_{S(G)} \Z[t^{-1}_G] \to \CH^*(X,q)[t^{-1}_G]
\]
is an isomorphism.
If $G$ is not necessarily split, 
then this map is an isomorphism with rational coefficients.
\end{thm}
\begin{proof}
We assume throughout the proof that we work with $\Z[t^{-1}_G]$-modules.
In view of \thmref{thm:SS-Red-Main-I}, we only need to show that
$E^2_{p,q} = 0$ for $p \neq 0$ in the spectral sequence ~\eqref{eqn:SS-M-01}.
We first assume that $G= T$ is a split torus. Let 
$C_{\bullet} \surj \CH^*(X, q)$ be a free resolution of
$\CH^*(X,q)$ as $\Z$-module. Since $S(T)$ is a polynomial algebra over
$\Z$, it follows from \thmref{thm:split} that $C_{\bullet} \otimes_{\Z} S(T)$ 
is a free resolution of $\CH^*_T(X,q)$ as $S(T)$-module. Using this
resolution, it follows immediately that $E^2_{p,q} = 0$ for $p \neq 0$.

We next let $G$ be a split reductive group with a split maximal torus $T$
and let $C_{\bullet} \surj \CH^*_G(X, q)$ be a free resolution of
$\CH^*_G(X,q)$ as $S(G)$-module. It follows from \thmref{thm:Sspl-M} that
$C_{\bullet} \otimes_{S(G)} S(T)$ is a free resolution of $\CH^*_T(X,q)$ as 
$S(T)$-module. Since the augmentation map $S(G) \to \Z$ factors through
$S(G) \to S(T) \to \Z$, it follows from the case of torus that
$E^2_{p,q}(G) = E^2_{p,q}(T) = 0$ for $p \neq 0$.

Suppose now that $G$ is a reductive group which is not necessarily
split over $k$.
By \cite[Expos{\'e}~XXII, Corollaire~2.4]{SGA3},
there is a finite Galois extension $l/k$ with Galois group $\Gamma$
such that $G_l$ is a split reductive group. Let $\CH^*_G\left(X_l, i \right)$
denote the equivariant higher Chow groups of $X_l$ for the action of
$G_l$.  It then follows from \cite[Lemma~3.3 and Remark~3.4]{KrishnaV} that 
$\Gamma$ acts on $\CH^*_G(X_l, i)$ and $\CH^*_G(X, i)$
is its Galois invariant. The same holds for $\CH^*(X, i)$.
Let ${tr}: \CH^*_G(X_l, i) \to \CH^*_G(X, i)$ denote 
the trace map 
\[
tr(x) = \frac{1}{|{\Gamma}|} {\underset{\gamma \in \Gamma}{\Sigma}} 
{\gamma}(x).
\]

Let $\ov{J_G \CH^*_G(X, i)}$ denote the image of 
$J_{G_l}\CH^*_G(X_l, i)$ under the trace map. Then we get a commutative diagram
\begin{equation}\label{eqn:SIR*}
\xymatrix@C1pc{
0 \ar[r] & J_G \CH^*_G(X, i) \ar[r] \ar@{^{(}->}[d] &
\CH^*_G(X, i) \ar[r] \ar@{^{(}->}[d] & \CH^*(X, i) 
\ar@{^{(}->}[d] \ar[r] & 0 \\
0 \ar[r] & J_{G_l} \CH^*_G(X_l, i) \ar[r] \ar[d] & 
\CH^*_G(X_l, i) \ar[r] \ar[d]_{tr} & \CH^*(X_l, i) \ar[r] 
\ar[d]^{tr} & 0 \\
& {\ov{J_G \CH^*_G(X, i)}} \ar@{^{(}->}[r] & 
\CH^*_G(X, i) \ar[r] & \CH^*(X, i). & }
\end{equation}

The composite vertical maps in the middle and the right columns are
the identity by the definition of the trace map. Hence the composite left 
vertical map is also the identity. The middle row is exact
since $G_l$ is split. A simple diagram chase now shows that the top row must 
also be exact, which proves the theorem.
\end{proof}

\section{Applications to equivariant $K$-theory}
\label{sect:EKT}
We shall now apply Theorems~\ref{thm:SS-Red-Main-I} and ~\ref{thm:Degn-*}
to answer some well known questions in equivariant $K$-theory of smooth 
projective schemes. Since all our results in this section will be with rational
coefficients, our convention will be that an abelian
group $A$ will actually mean the group $A {\otimes}_{\Z} {\Q}$.

\subsection{The forgetful map in $K$-theory}\label{sect:Foget-map}
To prove the surjectivity of the forgetful map from the equivariant to the
ordinary $K$-theory of schemes, we begin with the following result. 
For a linear algebraic group $G$, let $\widehat{R(G)}$ denote the $I_G$-adic 
completion of the ring $R(G)$. Similarly, we denote the $J_G$-adic completion
of the ring $S(G)$ by  $\widehat{S(G)}$. 
One of the consequences of the
Riemann-Roch theorem of \cite{ED1} is that there is a ring isomorphism 
\begin{equation}\label{eqn:EDT}
{ch}^G_k : \widehat{R(G)} \xrightarrow{\simeq}  \widehat{S(G)}.
\end{equation}

Recall from \S~\ref{sec:Intro} that for a linear algebraic group $G$ acting on 
a smooth scheme $X$, $\widehat{K^G_i(X)}$ denotes the $I_G$-adic completion of 
$K^G_i(X)$. Let $\widehat{\CH^*_G(X, i)}$
denote the $J_G$-adic completion of the $S(G)$-module
$\CH^*_G(X, i)$. We define the $\widehat{R(G)}$-module
$\wt{K^G_i(X)}$ to be the group $K^G_i(X) {\otimes}_{R(G)} 
\widehat{R(G)}$. We similarly define the $\widehat{S(G)}$-module 
$\wt{\CH^*_G(X, i)}$ to be the group
${\CH^*_G(X, i)} {\otimes}_{S(G)} \widehat{S(G)}$.
We recall from \cite[Theorem~1.2]{KrishnaV} that there
is a functorial Chern character map
\begin{equation}\label{eqn:KV}
{ch}^G_X : K^G_i(X) \to
\wt{\CH^*_G(X, i)}.
\end{equation}

Given a closed subgroup $H \subseteq G$, let
$\CH^*_G(X, i) \xrightarrow{s^G_H}
\CH^*_H(X, i)$ and $K^G_i(X) \xrightarrow{r^G_H} K^H_i(X)$
denote the restriction maps on equivariant Chow groups
and $K$-theory.

\begin{lem}\label{lem:Commute*}
Let $G$ be a linear algebraic group acting on a smooth quasi-projective
scheme $X$ over $k$ and let $H \subseteq G$ be a closed subgroup. 
Then there is a commutative  diagram
\[
\xymatrix@C2pc{
K^G_i(X) \ar[r]^{{ch}^G_X} \ar[d]_{r^G_H} &
{\wt{\CH^*_G(X, i)}} \ar[d]^{s^G_H} \\
K^H_i(X) \ar[r]_{{ch}^H_X} &
{\wt{\CH^*_H(X, i)}}.}
\] 
\end{lem}
\begin{proof}
By \cite[Corollary~3.10]{KrishnaV}, the map 
$\wt{\CH^*_G(X, i)} \to 
\stackrel{\infty}{\underset{j= 0}{\prod}} \CH^j_G(X, i)$
is injective. Hence we can replace the former by the latter group in the
above diagram. 
Now, it suffices to check that for any $\alpha \in
K^G_i(X)$ and any $j \ge 0$, the $j$-th components of $s^G_H \circ
{ch}^G_X (\alpha)$ and ${ch}^H_X \circ r^G_H (\alpha)$ are same
in the product $\stackrel{\infty}{\underset{j= 0}{\prod}} 
\CH^j_H(X, i)$. So we fix $j \ge 0$ and choose a good pair
$(V,U)$ for the $G$-action corresponding to $n \gg j$. Then $(V,U)$ is
a good pair for the $H$-action as well.
From the construction of the equivariant Chern character map
(see \cite[Proposition~4.1]{KrishnaV}), it suffices now to show that
the diagram
\[
\xymatrix{
K^G_i(X) \ar[r] \ar[d] & K_i(X \stackrel{G}{\times} U) \ar[r] \ar[d] &
\CH^*(X \stackrel{G}{\times} U, i) \ar[d] \\
K^H_i(X) \ar[r] & K_i(X \stackrel{H}{\times} U) \ar[r] &
\CH^*(X \stackrel{H}{\times} U, i)}
\]
commutes. 
Let us quickly recall the left horizontal arrows.
A $G$-equivariant vector bundle $\sF$ on $X$
pulls back to a $G$-equivariant vector bundle $\pi^*(\sF)$ on $X \times U$ 
via the projection $X \times U \xrightarrow{\pi} X$. 
The faithfully flat descent associated to the $G$-torsor
$X\times U \to X \stackrel{G}{\times} U$ then defines a unique
vector bundle $\pi^*{\sF}//G$ on $X \stackrel{G}{\times} U$.
This assignment at the level of the exact categories of vector bundles
defines the map $K^G_*(X) \to K_*(X \stackrel{G}{\times} U)$ on $K$-theory.

The left square in the above diagram clearly commutes, and the right square
commutes by the contravariance property of the non-equivariant Chern
character for the maps between smooth schemes.
\end{proof}

\begin{thm}\label{thm:forgetful*}
Let $G$ be a reductive group acting on a smooth 
projective scheme $X$ over $k$. Then for all $i \ge 0$, the forgetful map 
$K^G_i(X) \to K_i(X)$ induces an isomorphism
\begin{equation}\label{eqn:char}
{K^G_i(X)}/{I_GK^G_i(X)} \xrightarrow{\simeq} K_i(X)
\end{equation}
with rational coefficients. In particular, the natural map
${K^G_i(X)}_{\Q} \to {K_i(X)}_{\Q}$ is surjective.
\end{thm}
\begin{proof}
Applying Lemma~\ref{lem:Commute*} for $H = \{1\}$ and 
noting that the forgetful maps have factorizations
\[
K^G_i(X) \to K^G_i(X) {\otimes}_{R(G)} {\Q} \to K_i(X), 
\]
\[
{\wt{\CH^*_G(X, i)}} \to 
{\wt{\CH^*_G(X, i)}} {\otimes}_{\widehat{S(G)}} {\Q}
\to \CH^*(X, i),
\]
we get a diagram 
\begin{equation}\label{eqn:Commute*1}
\xymatrix{
K^G_i(X) \ar[r]^{{ch}^G_X} \ar@{->>}[d] &
{\wt{\CH^*_G(X, i)}} \ar@{->>}[d] \\ 
{K^G_i(X) {\otimes}_{R(G)} {\Q}} \ar[d]_{r^G_X} \ar@{.>}[r] &
{{\wt{\CH^*_G(X, i)}} {\otimes}_{\widehat{S(G)}} {\Q}}
\ar[d]^{s^G_X} \\
K_i(X) \ar[r]_{{ch}_X} &
{\CH^*(X, i)}.}
\end{equation}

Next we apply \cite[Theorem~1.2]{KrishnaV} to see that   
the map ${ch}^G_X$ of ~\eqref{eqn:KV} factors as
\begin{equation}\label{eqn:KV*}
K^G_i(X) \to \wt{K^G_i(X)} \xrightarrow{{\wt{ch}}^G_X}
{\wt{\CH^*_G\left(X, i \right)}}
\end{equation}
such that ${{\wt{ch}}^G_X}$ is an isomorphism of 
$\widehat{R(G)}$-modules.
In particular, ${{\wt{ch}}^G_X}$ and hence ${ch}^G_X$ takes
$I_GK^G_i(X)$ to $J_G {\wt{\CH^*_G(X, i)}}$. This 
shows that the map ${ch}^G_X$ in ~\eqref{eqn:Commute*1} descends to
a map 
\[
{K^G_i(X) {\otimes}_{R(G)} {\Q}} \xrightarrow{{\ov{ch}}^G_X}
{{\wt{\CH^*_G(X, i)}} {\otimes}_{\widehat{S(G)}} {\Q}}
\]
such that the above diagram commutes.

We claim that ${{\ov{ch}}^G_X}$ is an isomorphism. For this, we first note
that ${{\wt{ch}}^G_X}$ is an isomorphism of $\widehat{R(G)}$-modules
by \cite[Theorem~1.2]{KrishnaV}. The claim now follows by tensoring both 
sides of ${{\wt{ch}}^G_X}$ with ${\widehat{R(G)}}/{I_G{\widehat{R(G)}}}
\simeq {\widehat{S(G)}}/{J_G{\widehat{S(G)}}} \simeq {\Q}$ and then using 
the isomorphisms
\[
K^G_i(X) {\otimes}_{R(G)} {\Q} \xrightarrow{\simeq}
K^G_i(X) {\otimes}_{R(G)} ({R(G)/{I_G}}) 
\xrightarrow{\simeq} \wt{K^G_i(X)} {\otimes}_{\widehat{R(G)}} 
({\widehat{R(G)}/{I_G\widehat{R(G)}}}).
\]

It follows from \thmref{thm:Degn} and the isomorphism
\begin{equation}\label{eqn:SIR}
\CH^*_G(X, i) {\otimes}_{S(G)} {\Q} \xrightarrow{\simeq}
{{\wt{\CH^*_G(X, i)}} {\otimes}_{\widehat{S(G)}} {\Q}}
\end{equation}
that the map $s^G_X$ in the diagram ~\eqref{eqn:Commute*1} is an
isomorphism. The bottom horizontal map ${ch}_X$ in this diagram is an 
isomorphism by the non-equivariant Riemann-Roch theorem 
\cite[Theorem~9.1]{Bloch}. We conclude that the map $r^G_X$ is also
an isomorphism. In particular, the forgetful map $K^G_i(X) \to K_i(X)$
is surjective. This finishes the proof of the theorem.

\end{proof}

\begin{exm}\label{exm:Forget-2}
This example shows that the projectivity assumption is essential
in \thmref{thm:forgetful*}. Let $\G_m$ act on itself by multiplication.
Then the quotient map $\pi: \G_m \to {\G_m}/{\G_m}$ is same as the structure map
$\pi: \G_m \to \Spec(k)$. In particular, it follows from 
\cite[Proposition~A.1(6)]{KrishnaV} that the forgetful map
$K^{\G_m}_i(\G_m) \to K_i(\G_m)$ is same as the pull-back map
$\pi^*: K_i(k) \to K_i(\G_m)$. So it suffices to show that $\pi^*$ is not
surjective with $\Q$-coefficients. 

The open embedding $j: \G_m \inj \A^1_k$ gives us the localization 
sequence
\[
0 \to K_1(\A^1_k) \xrightarrow{j^*} K_1(\G_m) \to K_0(k) \to 0.
\]
The homotopy invariance shows that $j^*$ can be identified with the
pull-back map $\pi^*: K_1(k) \to K_1(\G_m)$. Any choice of a $k$-rational point
of $\G_m$ now shows that the above sequence is split exact.
We conclude that $\pi^*$ is not surjective since $K_0(k)_{\Q} \simeq \Q$.
\end{exm}

\subsection{Equivariant Riemann-Roch theorem}
\label{sect:ERR}
We now apply \thmref{thm:Degn} to generalize the equivariant
Riemann-Roch theorem of Edidin-Graham \cite[Theorem~4.1]{ED1}
to higher $K$-theory of smooth projective schemes.
In addition, we shall prove this Riemann-Roch theorem for all smooth
quasi-projective toric varieties.
Recall from \cite[Theorem~4.6]{KrishnaV} that for every $i \ge 0$, there is a 
Riemann-Roch map
\begin{equation}\label{eqn:ERR-*}
\tau^G_X: \widehat{K^G_i(X)} \to 
\stackrel{\infty}{\underset{j= 0}{\prod}} {\CH^j_G(X, i)}.
\end{equation}

\begin{thm}\label{thm:RRochP}
Let $G$ be a reductive group acting on a smooth quasi-projective
scheme $X$ over $k$. Assume one of the following.
\begin{enumerate}
\item
$X$ is projective.
\item
$X$ is a toric variety and $G$ is the dense torus of $X$.
\item
$G = \G_m$.
\end{enumerate}
Then the Riemann-Roch map $\tau^G_X$ is an isomorphism.
\end{thm}
\begin{proof}
By \cite[Corollary~1.3]{KrishnaV}, the Riemann-Roch map 
${\tau}^G_X$ has a factorization
\[
\widehat{K^G_i(X)} \to \widehat{\CH^*_G(X, i)}
\to \stackrel{\infty}{\underset{j= 0}{\prod}} {\CH^j_G(X, i)}
\]
and the first map is an isomorphism. Hence we only need to show that the
second map is also an isomorphism. By \cite[Proposition~3.2(3)]{KrishnaV},
we only need to show that $\CH^*_G(X, i)$ is generated as
$S(G)$-module by a (possibly infinite) set $S$ of homogeneous elements 
of bounded degree.

When $X$ is projective, \thmref{thm:Degn} shows that the map 
\[{\CH^*_G(X, i)}/{J_G \CH^*_G(X, i)}
\to \CH^*\left(X, i \right)
\]
is an isomorphism.
Since the latter is a graded module of bounded degree, 
say $d$, this isomorphism implies that
\[
\CH^*_G(X, i) \subset {\CH^*_G(X, i)}_{\le d}
+ J_G ({\CH^*_G(X, i)}_{\le d}) \subset S(G) 
({\CH^*_G(X, i)}_{\le d}).
\]

If $X$ is a smooth quasi-projective toric variety with dense torus $T$, it 
follows from \cite[Corollary~11.5]{Krishna1} that there is an $S(T)$-module 
isomorphism (with $\Q$-coefficients)
\[
\CH^*(k,i) \otimes_{\Q} \CH^*_T(X) \xrightarrow{\simeq} \CH^*_T(X,i).
\]
Since $\CH^*(k,i)$ is a graded $\Q$-module of bounded degree for all $i \ge 0$,
it follows from Corollary~\ref{cor:Chow0} 
(and the above argument for $X$ projective)
that the graded $S(T)$-module on the left of this isomorphism is
generated by homogeneous elements of bounded degree. It follows that the same
is true for $\CH^*_T(X,i)$

If $T$ is a rank one torus acting on any smooth scheme $X$, we apply
Corollary~\ref{cor:G-m-action} to conclude that $\CH^*_T(X,i) \otimes_{S(T)} \Q$
is an $S(T)$-module of bounded degree. We can now argue as in
the proof of case (1) to conclude that $\CH^*_T(X,i)$ is generated 
by homogeneous elements of bounded degree.
\end{proof}

\begin{remk}\label{remk:A-S-T}
Recall from \cite[\S~3]{Totaro} that if a linear algebraic group $G$ 
acts on a scheme $X$, then the Borel style equivariant algebraic 
$K$-theory of $X$ is defined by
\begin{equation}\label{eqn:topK}
K^{BS}_i(X_G) = \varprojlim K_i(X \stackrel{G} \times U),
\end{equation}
where the inverse limit is taken over the representations $V$ of $G$ such
that $(V,U)$ form good pairs. Using the non-equivariant Riemann-Roch for
the quotients $X\stackrel{G}{\times} U$ and the definition of 
$\CH^*_G(X,i)$, it can be easily checked that $K^{BS}_i(X_G)$ is
canonically isomorphic (over $\Q$) to 
$\stackrel{\infty}{\underset{j= 0}{\prod}} {\CH^j_G(X, i)}$.
Theorem~\ref{thm:RRochP} can then be rephrased as
saying that the natural map
\begin{equation}\label{eqn:A-S-T-0}
\wh{K^G_i(X)} \to K^{BS}_i(X_G)
\end{equation}
is an isomorphism for all $i \ge 0$ with rational coefficients.
 
This is an analogue of the Atiyah-Segal completion theorem \cite{AS}
for rational algebraic $K$-theory. For $X = \Spec(k)$ and $i =0$, this 
was shown earlier by Totaro \cite[Theorem~3.1]{Totaro}.
\end{remk}

\section{Equivariant $K$-theory with finite coefficients}
Theorem~\ref{thm:RRochP} shows that the completion of the rational
equivariant algebraic $K$-theory (with respect to the augmentation ideal)
of smooth schemes can be computed in terms of 
an equivariant algebraic cohomology, namely, the equivariant Chow groups.
We do not know if there is a Riemann-Roch type theorem for the
equivariant $K$-theory without going to this completion.
Our objective here is to give a partial answer to this question for
the equivariant $K$-theory with finite coefficients.
This answer is given in terms the equivariant Quillen-Lichtenbaum conjecture.
This conjecture provides an explicit relation between the 
algebraic and topological $K$-theories (with finite coefficients) of 
schemes over the field of complex numbers. We shall show that
the ordinary Quillen-Lichtenbaum conjecture also implies its
equivariant counterpart.

Let $G$ be a reductive group over $\C$ acting on a quasi-projective
$\C$-scheme $X$. We shall denote the analytic space $X(\C)$ by $X^{an}$.
Let us recall from \cite[\S~5.3]{Thomason1} (see also \cite{Segal}) that if 
$X$ is smooth, and if $M$ is a maximal compact subgroup of $G^{an}$,
the equivariant topological $K$-theory spectrum 
$K^M(X^{an})$ of $X^{an}$ is the spectrum of the classifying maps
of the equivariant topological vector bundles on $X^{an}$.
It is shown in \cite[\S~5.3]{Thomason1} that $K^M((-)^{an})$ 
extends to a Borel-Moore homology theory $G^M((-)^{an})$ 
on $\Sch^G_{\C}$. This homology theory satisfies the localization, 
homotopy invariance and Poincar{\'e} duality.
We shall denote the ordinary topological $K$-theory and $G$-theory of
$X^{an}$ by $K(X^{an})$ and $G(X^{an})$, respectively.

For a prime $p$, the equivariant topological
$G$-theory with coefficient ${\Z}/{p^{\nu}}$ is the smash product of
$G^{M}(X^{an})$ with the Moore spectrum 
${{\Sigma}^{\infty}}/{p^{\nu}}$. This will be denoted by 
$G^{M}(X^{an}; {\Z}/{p^{\nu}})$.     
The transition from the algebraic to the topological vector bundles
on smooth schemes induces a natural map 
${\rho}_X : G^G(X; {\Z}/{p^{\nu}}) \to
G^{M}(X^{an}; {\Z}/{p^{\nu}})$ for all quasi-projective schemes
(see \cite[Proposition~5.8]{Thomason1}). 
We wish to prove the following equivariant version of the 
Quillen-Lichtenbaum conjecture.

\begin{thm}\label{thm:QL}
Let $G$ be a reductive group acting on a smooth quasi-projective
scheme $X$ of dimension $d$ over $\C$. Let $M$ be a maximal compact 
subgroup of the Lie group $G^{an}$. Then the natural map
\[
{\rho}_X : K^G_i(X; {\Z}/{p^{\nu}}) \to
K^M_i(X^{an}; {\Z}/{p^{\nu}})
\]
is an isomorphism for $i \ge d-1$ and a monomorphism for $i = d-2$.
\end{thm}

\vskip .3cm

We write ${\Lambda} = {\Z}/{p^{\nu}}$ and $R(G; \Lambda) =
{{R(G)}/{p^{\nu}}}$ for the rest of this section. We need the following
result about the ordinary $K$-theory to prove the above theorem.

\begin{lem}\label{lem:SQL}
Let $X$ be a $\C$-scheme of dimension $d$. Then the natural map
\[
G_i(X; \Lambda) \to G_i(X^{an}; \Lambda)
\]
is an isomorphism for $i \ge d-1$ and a monomorphism for $i = d-2$.
\end{lem}
\begin{proof} 
We prove the lemma by induction on $\dim(X)$.
First suppose that $X$ is smooth. It was shown by Suslin 
\cite[Theorem~4.1]{PW} in this case that the lemma 
follows from the Bloch-Kato conjecture, established  by 
Voevodsky and Rost \cite{Voev} .
If $X$ is singular, the isomorphism $G_*(X) \simeq G_*(X_{\rm red})$ 
allows one to assume $X$ to be reduced. 

Assume that the lemma holds for schemes of dimension 
less than $d$. If $X = X_1 \cup \cdots \cup X_r$ is a union of its irreducible
components, we let $Z = X_2 \cup \cdots \cup X_r$ and $U = X \setminus Z$. 
Then $U$ is irreducible and $Z$ has smaller number of irreducible components.
Now, the localization sequences (use \cite[\S~1, 5.3]{Thomason1}
with $G = \{1\}$) 
\begin{equation}\label{eqn:OQL}
\xymatrix@C.5pc{
G_{i+1}(U, \Lambda) \ar[r] \ar[d] &
G_i(Z, \Lambda) \ar[r] \ar[d] & 
G_i(X, \Lambda) \ar[r] \ar[d] &
G_i(U, \Lambda) \ar[r] \ar[d] &
G_{i-1}(Z, \Lambda) \ar[d] \\
G_{i+1}(U^{an}, \Lambda) \ar[r] &
G_i(Z^{an}, \Lambda) \ar[r] & 
G_i(X^{an}, \Lambda) \ar[r] &
G_i(U^{an}, \Lambda) \ar[r] &
G_{i-1}(Z^{an}, \Lambda)}
\end{equation}
of the algebraic and topological $G$-theories and an induction on the
number of irreducible components reduces the problem to the case when
$X$ is irreducible. We alert the reader here that the case $i = d-1,
d-2$ needs a more careful use of snake lemma in the above diagram but
causes no trouble. 

If $X$ is irreducible, there exists a dense open subset $U \subset X$ which
is smooth and hence the complement $Z = X \setminus U$ has dimension smaller
than $d$. The diagram ~\eqref{eqn:OQL} above and induction on dimension
together with the smooth case finish the proof.
\end{proof}

We recall the following construction before we prove the next result.
If $G$ is a linear algebraic group acting trivially on a scheme $X$,
let ${Coh}_X$ (resp. ${Coh}^G_X$) denote the abelian category of 
coherent (resp. $G$-equivariant coherent) ${\sO}_X$-modules. Then there is a
natural exact functor ${Coh}_X \to {Coh}^G_X$ by letting $G$ act trivially
on a coherent sheaf $\sF$. This gives a natural map of spectra
$G_*(X) \xrightarrow{f} G^G_*(X)$. Since $G^G_*(X)$ is a module spectrum
over the representation ring $R(G)$, we obtain a natural map
\begin{equation}\label{eqn:TR}
G_*(X) {\otimes}_{\Z} R(G) \to G^G_*(X)
\end{equation}
\[
{\alpha} {\otimes} {\rho} \mapsto {\rho} \cdot f(\alpha).
\]
We also have a similar map on the $K$-theory with finite coefficients.

\begin{lem}\label{lem:simple}
Let $T$ be a diagonalizable group acting trivially on a $\C$-scheme $X$. 
Then the maps
\begin{equation}\label{eqn:QL0}
G_i(X; \Lambda) {\otimes}_{\Lambda} R(T; \Lambda)
\to  G^T_i(X; \Lambda), \ {\rm and} 
\end{equation}
\[
G_i(X^{an}; \Lambda) {\otimes}_{\Lambda} R(T; \Lambda)
\to G^T_i(X^{an}; \Lambda) 
\]
are isomorphisms of $R(T; \Lambda)$-modules.
\end{lem}
\begin{proof} We consider the following commutative diagram of short exact 
sequences of
$\Lambda$-modules.
\begin{equation}\label{eqn:QL1}
\xymatrix@C.5pc{
0 \ar[r] & {{G_i(X)}/{p^{\nu}}} {\otimes}_{\Lambda} R(T; \Lambda) \ar[r]
\ar[d] & G_i(X; \Lambda) {\otimes}_{\Lambda} R(T; \Lambda)
\ar[d] \ar[r] & {\rm Tor}^1_{\Z} (G_{i-1}(X), \Lambda) 
{\otimes}_{\Lambda} R(T; \Lambda) \ar[r]
\ar[d] & 0 \\
0 \ar[r] & {{G^T_i(X)}/{p^{\nu}}} \ar[r] & 
G^T_i(X; \Lambda) \ar[r] &
{\rm Tor}^1_{\Z} (G^T_{i-1}(X), \Lambda) \ar[r] & 0.}
\end{equation}

The top row is obtained by tensoring the non-equivariant version of the
bottom row with $R(T)$. This row is then exact since $R(T)$ is a flat 
$\Z$-module.
The left vertical arrow is an isomorphism by \cite[Lemma~5.6]{Thomason3}.
Moreover, the flatness of $R(T)$ also implies that the last term on the
bottom row is the same as ${\rm Tor}^1_{\Z} 
(G_{i-1}(X) {\otimes}_{\Z} R(T), \Lambda)
\simeq {\rm Tor}^1_{\Z} (G_{i-1}(X), \Lambda)
{\otimes}_{\Z} R(T)$. In particular, the right vertical arrow is an 
isomorphism and hence so is the middle vertical map.
The second isomorphism in ~\eqref{eqn:QL0} follows exactly in the same way
once we know that the left vertical map in the topological version of
the diagram ~\eqref{eqn:QL1} is an isomorphism. But this follows from
\cite[Proposition~2.2]{Segal}. 
\end{proof}

\vskip .3cm

{\sl{Proof of \thmref{thm:QL}:}}
We in fact prove the more general statement that for any $\C$-scheme
$X$ of dimension $d$ with $G$-action, the map
\begin{equation}\label{eqn:SQL*}
{\rho}_X : G^G_i(X; {\Z}/{p^{\nu}}) \to
G^M_i(X^{an}; {\Z}/{p^{\nu}})
\end{equation}
is an isomorphism for $i \ge d-1$ and a monomorphism for $i = d-2$.

Following an idea of Thomason, we prove it by reducing  to the case of 
torus. So we first assume that $G = T$ is a torus which acts trivially on $X$. 
In this case, the proposition follows directly from Lemmas ~\ref{lem:SQL}
and~\ref{lem:simple} using the additional known fact that $R(T) \simeq R(M)$.

If $T$ does not act trivially on $X$, we prove the proposition by 
induction on $\dim(X)$. We use Thomason's 
generic slice theorem \cite[Proposition~2.4]{Thomason1}
to get a dense open $T$-invariant open set $U \subset X$ and a
diagonalizable subgroup $T_1 \subset T$ with quotient $T_2 = T/{T_1}$ such 
that $T$ acts on $U$ via $T_2$, which in turn acts freely on $U$ such that
$U/T$ is a geometric quotient and there is a $T$-equivariant isomorphism
\[
U \xrightarrow{\simeq} (U/T) \times T_2 \xrightarrow{\simeq}
(U/T) \stackrel{T_1} \times T.
\]
In particular, we have $G^T(U ; \Lambda)
\simeq G^T(U/T \stackrel{T_1} \times T ; \Lambda) 
\simeq G^{T_1}(U/T ; \Lambda)$,
where the last homotopy equivalence follows from the Morita equivalence
\cite[Theorem~1.10]{Thomason1}.

The topological version of the Morita equivalence follows from
\cite[Proposition~2.1]{Segal} and the example preceding this in
\cite{Segal} (see also \cite[\S~5.3]{Thomason1}). 
In particular, we get $G^M(U^{an} ; \Lambda)
\simeq G^{M_1}((U/T)^{an} ; \Lambda)$.
Since $T_1$ acts trivially on $U/T$, the theorem holds for the
$T_1$-action on $U/T$. Since ${\rm dim}(U) \ge {\rm dim}(U/T)$,
we conclude from the above Morita isomorphisms that the theorem
holds also for the $T$-action on $U$.

Now we set $Z = X \setminus U$, use the diagram ~\eqref{eqn:OQL} of 
localization sequences and the induction and argue as before to conclude the
proof of the theorem when $G$ is a torus.  

To prove general case, choose a maximal torus $T$ of $G$ and a Borel 
subgroup $B$ of $G$ containing $T$.
Choose $M$ as a maximal compact subgroup of $G$ containing a maximal compact 
subgroup $M'$ of $T$. Then the compactness of $G/B$ implies that
$G/B \simeq M/{M'}$.
We thus have the maps of spectra 
\[
G^G(X ; \Lambda) \to G^T(X ; \Lambda)
\xrightarrow{\simeq} G^G(G/B \times X ; \Lambda)
\to G^G(X ; \Lambda)
\]
\[
G^M(X^{an} ; \Lambda) \to 
G^{M'}(X^{an} ; \Lambda)
\xrightarrow{\simeq} G^{M}(M/{M'} \times X^{an} ; \Lambda)
\to G^M(X^{an} ; \Lambda)
\]
such that they are compatible under ${\rho}_X$ and both the composite
maps are identity (see \cite[Proof of Theorem~5.9]{Thomason1}). 
The theorem now follows from the case of torus.
$\hfill \square$

\vskip .4cm

\noindent\emph{Acknowledgements.} 
The author would like to thank the referee for very carefully reading the
paper and providing many suggestions to improve its contents and 
presentation.

\enlargethispage{20pt}

\end{document}